\documentclass[reqno,12pt]{amsart}

\usepackage{amsmath,amssymb,amsthm}
\usepackage{bbm,mathtools,mathrsfs}
\usepackage[margin=1in]{geometry}
\usepackage[hidelinks,bookmarksnumbered=true,bookmarksdepth=section]{hyperref}
\hypersetup{
pdftitle={Primitive Fourier Coefficients of Hermitian Klingen--Eisenstein Series},
pdfauthor={Nobuki Takeda}
}

\allowdisplaybreaks
\makeatletter
\@namedef{subjclassname@2020}{\textup{2020} Mathematics Subject Classification}
\makeatother

\title{Primitive Fourier Coefficients of Hermitian Klingen--Eisenstein Series}
\author[N. Takeda]{Nobuki Takeda}
\address{Department of Mathematics, Graduate School of Science\\
Kyoto University,
Kyoto 606-8502,
Japan}
\email{takeda.nobuki.z04@kyoto-u.jp}
\makeatletter
\let\ams@tocwrite\@tocwrite
\def\@tocwrite#1#2{%
  \ifx\@secnumber\@empty
  \else
    \ams@tocwrite{#1}{#2}%
  \fi
}
\makeatother
\theoremstyle{definition}
\newtheorem{dfn}{Definition}[section]
\newtheorem{rem}[dfn]{Remark}
\newtheorem{ex}[dfn]{Example}
\theoremstyle{plain}
\newtheorem{thm}[dfn]{Theorem}
\newtheorem*{introthm}{Theorem}
\newtheorem{prop}[dfn]{Proposition}
\newtheorem{lem}[dfn]{Lemma}
\newtheorem{cor}[dfn]{Corollary}

\DeclareMathOperator{\tr}{tr}
\DeclareMathOperator{\Tr}{Tr}

\ExplSyntaxOn
\clist_const:Nn \c_alphabet_clist
{A,B,C,D,E,F,G,H,I,J,K,L,M,N,O,P,Q,R,S,T,U,V,W,X,Y,Z,
a,b,c,d,e,f,g,h,i,j,k,l,m,n,o,p,q,r,s,t,u,v,w,x,y,z}
\cs_new:Npn \makealphabet #1 #2
{\clist_map_inline:Nn \c_alphabet_clist {\exp_args:Nc \newcommand {#1##1}{#2{##1}}}}
\makealphabet{bb}{\mathbb}
\makealphabet{bf}{\mathbf}
\makealphabet{cal}{\mathcal}
\makealphabet{frak}{\mathfrak}
\makealphabet{scr}{\mathscr}
\makealphabet{rm}{\mathrm}
\clist_map_inline:nn
{GL,U,Her,M,St,rank,vol,hol,pr,Gal,prim,Kl,cusp,Si,Lie,Aut,covol,Sym,SL,diag,pol,Span,ord} {\exp_args:Nc \newcommand {#1}{\mathrm{#1}}}
\ExplSyntaxOff

\numberwithin{equation}{section}
\begin{document}
\begin{abstract}
  Let \(E/F\) be a CM extension and let \(1\le r<n\).
  We determine primitive Fourier coefficients of Hermitian Klingen--Eisenstein series on \(\U_{n,n}\) induced from cusp forms on \(\U_{r,r}\), at principal congruence level.
  Using the pullback formula for Hermitian Eisenstein series, we express a normalized primitive Fourier coefficient of a Klingen--Eisenstein series as the corresponding primitive Fourier coefficient of a Siegel Eisenstein series multiplied by a Petersson inner product with an explicitly constructed holomorphic automorphic form.
  The automorphic form in the Petersson inner product is a finite sum of products of Hermitian theta series and Siegel Eisenstein series of degree \(r\).
  Under some hypotheses, we also prove the \(p\)-integrality of these normalized primitive Fourier coefficients.
\end{abstract}
\maketitle
\tableofcontents
\section{Introduction}

Klingen--Eisenstein series are a natural source of arithmetic information on Fourier coefficients and congruences between Eisenstein and cuspidal automorphic forms.
Pullback formulas provide a way to relate their Fourier coefficients to those of Siegel Eisenstein series,
following the work of Garrett and Shimura \cite{Garrett1984Pullbacks,Shimura1997EulerProducts}.
For Hermitian modular forms,
the differential operators and pullback formulas used here were established in \cite{Takeda2025DifferentialOperators,Takeda2026PullbackFormula},
and their arithmetic applications were developed in \cite{Takeda2026Congruences}.
Related integrality questions for Klingen--Eisenstein series have been studied in \cite{mizumoto1996on,mizumoto2005congruences}.

The purpose of this paper is to determine primitive Fourier coefficients of Hermitian Klingen--Eisenstein series at principal congruence level over a general CM extension.
For Siegel modular forms,
primitive coefficients can be defined by M\"obius inversion \cite{BoechererRaghavan1988,BoechererSchulzePillot1991}.
Over a general CM field,
the lattices associated with cusp representatives need not be free,
so we work with projective \(\calO_E\)-lattices.
M\"obius inversion then gives primitive Fourier coefficients independent of these choices.
Applied to the pullback formula,
it cancels the contributions from larger lattices and reduces the relevant singular Fourier coefficients to primitive degree-\(n\) Siegel coefficients.

Let \(E/F\) be a CM extension with \(F\) totally real,
let \(1\le r<n\),
and put \(G_m=\U_{m,m}\).
Let \(\frakN\subset\calO_F\) be prime to \(2\frakD_{E/F}\),
and let \(f\in\calS_r(\rho_r,\frakN)\) generate an irreducible cuspidal automorphic representation.
We use the weights and the Hecke character \(\chi\) fixed in Section~\ref{sec:automorphic-setup},
and write \(\mu\) for the common value of \(\kappa_v+\nu_v\).

For a cusp representative \(\gamma\) and \(T\in\Lambda_\gamma(\frakN)_{>0}\),
write \(A_{\Kl}^{\prim}(\gamma,T)\) and \(A_{\Si}^{\prim}(\gamma,T)\) for the primitive Fourier coefficients of the Klingen--Eisenstein series and the normalized degree-\(n\) Siegel Eisenstein series, respectively, as defined in Subsection~\ref{subsec:primitive-formula}.
Let \(C_{n+r,\kappa,\nu}^{(\frakN)}(f;\chi)\) be the constant in the global pullback formula of Theorem~\ref{thm:global-pullback}, and put
\[
  c_f=\frac{C_{n+r,\kappa,\nu}^{(\frakN)}(f;\chi)}{(f^\dagger,f^\dagger)_{G_r}}.
\]
The main result,
proved in Theorem~\ref{thm:main-petersson-formula},
is the following.
\begin{introthm}
  Assume \(\mu\ge n+r\).
  If \(\mu=n+r+1\) and \(F=\bbQ\), assume
  \(\chi_0\ne\chi_{E/F}^{n+r+1}\).
  If \(\mu=n+r\) and \(n=r+1\), assume that
  \(L^{(\frakN)}(1,f,\St)\) is holomorphic and nonzero.
  Then
  \[
    c_f A_{\Kl}^{\prim}(\gamma,T)
    =
    A_{\Si}^{\prim}(\gamma,T)
    \frac{(\Phi_{\gamma,T},f^\dagger)_{G_r}}
    {(f^\dagger,f^\dagger)_{G_r}}.
  \]
  Here \(\Phi_{\gamma,T}\) is the \(V_{\sigma_n}\otimes V_{\rho_r'}\)-valued holomorphic automorphic form defined in Subsection~\ref{subsec:theta-eisenstein}.
\end{introthm}
On each adelic component,
\(\Phi_{\gamma,T}\) is a finite sum of products of a degree-\(r\) Hermitian theta series and a degree-\(r\) Siegel Eisenstein series.
The pullback identity is first established in a region of absolute convergence and then continued to the point required in the theorem.
The Petersson integrals at the continued value are absolutely convergent by the rapid decay of cusp forms.

The formula also has arithmetic consequences.
After fixing integral structures,
we obtain \(p\)-integrality results for normalized
Klingen--Eisenstein Fourier coefficients and,
under additional local hypotheses,
for their primitive Fourier coefficients.
In degree two,
the explicit Fourier expansion of \(\Phi_{\gamma,T}\)
gives a criterion for the relevant Petersson inner product to be nonzero
modulo \(\frakP\).
Together with the formula above,
this can be used to verify the Fourier-coefficient condition occurring in Harder-type congruences \cite{katsurada2026harder}.

Section~\ref{sec:automorphic-setup} fixes the automorphic forms,
Eisenstein series,
and differential operators used in the pullback formula.
Section~\ref{sec:pullback} proves the local and global pullback formulas and establishes the required continuation.
Section~\ref{sec:primitive} defines primitive Fourier coefficients and proves the primitive Fourier coefficient formula.
Section~\ref{sec:arithmetic-applications} gives the integrality results,
the degree-two rank criterion,
and the numerical applications.
\subsection*{Acknowledgements}
The author is grateful to T. Ikeda for his guidance and support,
and to H. Katsurada for kindly sharing his notes on Fourier coefficients of Klingen--Eisenstein lifts in the case of Siegel modular forms.
This work was supported by the Japan Society for the Promotion of Science KAKENHI [grant number JP26KJ1378].
\subsection*{Notation}

Throughout the paper,
\(F\) is a totally real field of degree \(d_F=[F:\bbQ]\),
and \(E/F\) is a CM extension.
Let \(\calO_F\) and \(\calO_E\) be their rings of integers.
We write \(x\mapsto\overline{x}\) for the nontrivial element of
\(\Gal(E/F)\),
and \(\chi_{E/F}\) for the associated quadratic Hecke character of \(F\).
Let \(\frakd_F\) be the different ideal of \(F/\bbQ\),
and let \(\frakD_{E/F}\subset\calO_F\) be the relative discriminant ideal of \(E/F\).
For an extension \(K/k\),
we denote its norm and trace maps by \(N_{K/k}\) and \(\Tr_{K/k}\),
respectively.

For a number field \(K\) and a nonzero prime ideal
\(\frakq\subset\calO_K\),
let \(v_{\frakq}:K^\times\to\bbZ\) be the normalized additive valuation,
so that \(v_{\frakq}(x)=\ord_{\frakq}(x)\),
and put \(v_{\frakq}(0)=+\infty\).
We write \(\calO_{K,\frakq}\) for the localization of \(\calO_K\) at
\(\frakq\).
For \(x,y\in K^\times\), write \(x\sim_{\frakq}y\) if
\(x/y\in\calO_{K,\frakq}^{\times}\).
For a rational prime \(p\),
\(v_p\) denotes the usual \(p\)-adic valuation normalized by \(v_p(p)=1\).

Let \(\bfh\) and \(\bfa\) denote the finite and infinite places of \(F\).
For a place \(v\) of \(F\), put \(E_v=E\otimes_FF_v\).
We write
\(\bbA_F=\bbA_{F,\bfh}\times\bbA_{F,\infty}\) and
\(\bbA_E=\bbA_{E,\bfh}\times\bbA_{E,\infty}\),
where
\(\bbA_{F,\bfh}={\prod'_{v\in\bfh}}F_v\) and
\(\bbA_{F,\infty}=\prod_{v\in\bfa}F_v\),
and similarly for \(E\).
For \(v\in\bfh\),
let \(\frakp_v\) be the corresponding prime ideal of \(\calO_F\),
let \(\calO_{F_v}\) be the ring of integers of \(F_v\),
and put
\(\calO_{E,v}=\calO_E\otimes_{\calO_F}\calO_{F_v}\).
Fix a uniformizer \(\varpi_v\) of \(F_v\),
and put \(q_v=\#(\calO_{F_v}/\varpi_v\calO_{F_v})\).
We use the normalized absolute value \(|\cdot|_v\) on \(F_v\),
with \(|\varpi_v|_v=q_v^{-1}\) at finite places and the usual real
absolute value at infinite places.

For a ring \(R\),
let \(\M_{a,b}(R)\) denote the space of \(a\times b\) matrices over \(R\),
write \(\M_a(R)=\M_{a,a}(R)\),
and let \(I_a\) be the identity matrix of degree \(a\).
For a matrix \(X\) over \(E\), put \(X^*={}^t\overline X\).
For a Hermitian matrix \(T\) and a matrix \(X\) of compatible size,
write \(T[X]=X^*TX\).
We denote by \(\Her_m(E)\) the space of Hermitian matrices in \(\M_m(E)\),
and use \(\tr\) and \(\det\) for matrix trace and determinant.
For an archimedean tuple \(X=(X_v)_{v\in\bfa}\),
put \(\tr_\infty(X)=\sum_{v\in\bfa}\tr(X_v)\).
For \(T\in\Her_m(E)\),
write \(T>0\) (resp. \(T\ge0\)) if \(T_v\) is positive definite
(resp. positive semidefinite) for every \(v\in\bfa\).
For a subset \(X\subset\Her_m(E)\),
put \(X_{>0}=\{T\in X:T>0\}\).

Let
\(\psi_{\bbQ}:\bbA_{\bbQ}/\bbQ\to\bbC^\times\)
be the standard additive character,
and put
\(\psi_F=\psi_{\bbQ}\circ\Tr_{F/\bbQ}:
\bbA_F/F\to\bbC^\times\).
We write \(\psi_{F_v}\) for its local component at \(v\).
Fourier transforms are taken with respect to the self-dual Haar measures
for the specified pairings.
For the archimedean integrals,
we use product Lebesgue measure on \(\Her_m(\bbC)\)
in the diagonal entries and in the real and imaginary parts of the entries
above the diagonal,
and the corresponding product measure over \(v\in\bfa\).
When \(m=r\), we denote this product measure by \(dV\).

\section{Preliminaries}\label{sec:automorphic-setup}

\subsection{Unitary groups and automorphic forms}\label{subsec:unitary-groups}

For a positive integer \(m\),
put
\[
  J_m=\begin{pmatrix}0&-I_m\\ I_m&0\end{pmatrix}.
\]
We denote by \(G_m=\U_{m,m}\) the quasi-split unitary group over \(F\) defined by
\[
  G_m(R) =\left\{g\in\GL_{2m}(E\otimes_F R)\mid g^*J_mg=J_m\right\}
\]
for every \(F\)-algebra \(R\).
For \(Y\in\Her_m(E)\),
put
\[
  n_m(Y)= \begin{pmatrix} I_m & Y \\ 0 & I_m \end{pmatrix}\in G_m(F),
\]
and use the same notation over \(F_v\) and \(\bbA_F\).
For \(v\in\bfh\),
put
\[
  K_{m,v}=G_m(F_v)\cap\GL_{2m}(\calO_{E,v}),
\]
and
\[
  K_{m,\bfh}=\prod_{v\in\bfh}K_{m,v}.
\]
For a nonzero ideal \(\fraka_v\subset\calO_{F_v}\),
define
\[
  K_{m,v}(\fraka_v) =\{k\in K_{m,v}\mid k\equiv I_{2m}\pmod{\fraka_v\calO_{E,v}}\},
\]
and put \(K_{m,v}(\calO_{F_v})=K_{m,v}\).
For a nonzero integral ideal \(\frakN=\prod_{v\in\bfh}\frakp_v^{c_v}\) of \(\calO_F\),
put
\[
  K_m(\frakN) =\prod_{v\in\bfh}K_{m,v}(\frakp_v^{c_v}).
\]
Thus \(K_m(\calO_F)=K_{m,\bfh}\).
For every \(v\in\bfh\),
write \(\frakN_v=\frakp_v^{c_v}\).

At every \(v\in\bfa\),
the corresponding archimedean factor is identified with the real Lie group
\[
  \U(m,m) =\left\{g\in\GL_{2m}(\bbC)\middle|g^*J_mg=J_m\right\}.
\]
Put
\[
  \calH_m =\left\{Z\in\M_m(\bbC)\middle| \frac{Z-Z^*}{2\sqrt{-1}}>0\right\}.
\]
Under the identification \(G_m(\bbA_{F,\infty})\simeq \prod_{v\in\bfa}\U(m,m)\),
the factor at \(v\) acts on \(\calH_m\) by
\[
  \begin{pmatrix}A&B\\C&D\end{pmatrix}Z =(AZ+B)(CZ+D)^{-1}.
\]
Let \(K_{m,v}\) be the stabilizer of \(\sqrt{-1}I_m\).
Then \(K_{m,v}\simeq\U(m)\times\U(m)\).
We denote its complexification by \(K_{m,v}^{\bbC}\simeq\GL_m(\bbC)\times\GL_m(\bbC)\).
Put
\[
  K_{m,\infty} =\prod_{v\in\bfa}K_{m,v}, \qquad K_{m,\infty}^{\bbC} =\prod_{v\in\bfa}K_{m,v}^{\bbC}.
\]
If \(\frakg_{m,v}=\Lie(\U(m,m))\),
we write
\[
  \frakg_{m,v}^{\bbC} =\frakg_{m,v}\otimes_\bbR\bbC =\frakk_{m,v}^{\bbC} \oplus\frakp_{m,v}^{+} \oplus\frakp_{m,v}^{-}
\]
for the complexified Cartan decomposition,
where \(\frakk_{m,v}^{\bbC}=\Lie(K_{m,v})\otimes_{\bbR}\bbC\).
Under the infinitesimal action on \(\calH_m\) at \(\sqrt{-1}I_m\),
the subspaces \(\frakp_{m,v}^{+}\) and \(\frakp_{m,v}^{-}\) correspond to the holomorphic and antiholomorphic tangent directions,
respectively.
Put
\[
  \frakp_{m,\infty}^{-} =\bigoplus_{v\in\bfa}\frakp_{m,v}^{-}.
\]

Throughout this paper,
we consider only dominant integral weights \(\lambda=(\lambda_1,\ldots,\lambda_m)\) of \(\GL_m(\bbC)\) satisfying \(\lambda_1\ge\cdots\ge\lambda_m\ge0\).
We denote by \(\rho_{m,\lambda}\) the irreducible polynomial representation of \(\GL_m(\bbC)\) with highest weight \(\lambda\).
If \(\lambda=(\lambda_1,\ldots,\lambda_s)\) with \(s<m\),
we identify it with \(\lambda=(\lambda_1,\ldots,\lambda_s,0,\ldots,0)\).
We put
\begin{align*}
  |\lambda| & =\sum_{i=1}^m\lambda_i, \\ \ell(\lambda) & =\max\{i\mid \lambda_i\ne0\},
\end{align*}
with \(\ell(0)=0\).
We fix integers \(1\le r<n\).
For every \(v\in\bfa\),
we fix positive integers \(\kappa_v,\nu_v\) and dominant integral weights \(\bfk_v,\bfl_v\) of \(\GL_r(\bbC)\) such that \(\kappa_v+\nu_v=\mu\) and \(\ell(\bfk_v)+\ell(\bfl_v)\le\mu\),
where the same integer \(\mu\) is used for all \(v\in\bfa\).
Write \(\bfk_v=(k_{v,1},\ldots,k_{v,r})\) and \(\bfl_v=(l_{v,1},\ldots,l_{v,r})\).
For \(m\ge r\),
we regard \(\bfk_v\) and \(\bfl_v\) as dominant integral weights of \(\GL_m(\bbC)\) by appending zeros.
We write \(\kappa=(\kappa_v)_{v\in\bfa}\),
\(\nu=(\nu_v)_{v\in\bfa}\),
\(\bfk=(\bfk_v)_{v\in\bfa}\),
and \(\bfl=(\bfl_v)_{v\in\bfa}\).

For \(m\ge r\),
put
\[
  \sigma_m = \boxtimes_{v\in\bfa} ( \rho_{m,\bfk_v}\boxtimes\rho_{m,\bfl_v} ), \qquad \sigma_m' = \boxtimes_{v\in\bfa} ( \rho_{m,\bfl_v}\boxtimes\rho_{m,\bfk_v} ).
\]
For tuples \(\alpha=(\alpha_v)_{v\in\bfa}\) and \(\beta=(\beta_v)_{v\in\bfa}\) of integers,
put
\[
  \tau_{\alpha,\beta} = \boxtimes_{v\in\bfa} ( {\det}^{\alpha_v}\boxtimes{\det}^{\beta_v} ).
\]
In particular,
we use \(\tau_{\kappa,\nu}\) for the fixed weight representation above,
and define
\[
  \rho_m = \tau_{\kappa,\nu}\otimes\sigma_m, \qquad \rho_m' = \tau_{\kappa,\nu}\otimes\sigma_m'.
\]
Since \(\tau_{\kappa,\nu}\) is one-dimensional,
we identify \(V_{\rho_m}\) with \(V_{\sigma_m}\) and \(V_{\rho_m'}\) with \(V_{\sigma_m'}\) as vector spaces.
The factor \(\tau_{\kappa,\nu}\) is retained in all transformation laws.
For \(v\in\bfa\),
we write \(\sigma_{m,v}\),
\(\sigma'_{m,v}\),
\(\rho_{m,v}\),
and \(\rho'_{m,v}\) for the \(v\)-components of these representations.

Fix an unramified unitary Hecke character \(\chi:E^\times\backslash\bbA_E^\times\rightarrow\bbC^\times\) with infinity type
\[
  \chi_{\bfa}((z_v)_{v\in\bfa}) =\prod_{v\in\bfa} |z_v|^{-\kappa_v+\nu_v} z_v^{\kappa_v-\nu_v}.
\]
Recall that \(\chi_0=\chi|_{\bbA_F^\times}\).
At the archimedean places,
\(\chi_0\) is a sign character.
Since \(\chi\) is unramified at the finite places,
it follows that \(\chi_0\) is of finite order.

For a finite-dimensional algebraic representation \((\rho,V_\rho)\) of \(K_{m,\infty}^{\bbC}\)
and a compact open subgroup \(K^\bfh\subset G_m(\bbA_{F,\bfh})\),
let \(\calA_m(\rho,K^\bfh)\) denote the space of holomorphic Hermitian automorphic forms of weight \(\rho\) and level \(K^\bfh\),
and let \(\calS_m(\rho,K^\bfh)\) be its cuspidal subspace.
Explicitly,
an element is left \(G_m(F)\)-invariant and right \(K^\bfh\)-invariant,
satisfies \(f(gk_\infty)=\rho(k_\infty)^{-1}f(g)\) for \(k_\infty\in K_{m,\infty}\),
and is annihilated by \(\frakp_{m,\infty}^{-}\).
We also impose moderate growth and,
in degree one,
holomorphy at the cusps.

For a nonzero integral ideal \(\frakN\subset\calO_F\),
we abbreviate
\[
  \calA_m(\rho,\frakN) =\calA_m(\rho,K_m(\frakN)), \qquad \calS_m(\rho,\frakN) =\calS_m(\rho,K_m(\frakN)).
\]

Normalize Haar measure \(dg_{m,v}\) on \(G_m(F_v)\) by \(\vol(K_{m,v})=1\) for \(v\in\bfh\).
For \(v\in\bfa\),
use Haar probability measure on \(K_{m,v}\) and quotient measure \(\det(Y_v)^{-2m}\,dX_v\,dY_v\) on \(G_m(F_v)/K_{m,v}\simeq\calH_m\).
Let \(d\bar g_m\) be the measure induced by \(dg_m=\prod_v dg_{m,v}\) on \(G_m(F)\backslash G_m(\bbA_F)\),
with counting measure on \(G_m(F)\).

For every finite-dimensional algebraic representation \(\rho\) of \(K_{m,\infty}^{\bbC}\) occurring below,
fix a \(K_{m,\infty}\)-invariant positive definite Hermitian inner product
\(
\langle\ ,\ \rangle_\rho
\)
on \(V_\rho\),
linear in the first variable and conjugate-linear in the second.
For the bideterminant realizations occurring in the differential operator and the pullback formula,
we use the inner product
\[
  \langle P,Q\rangle
  =\left.P(\partial_U,\partial_V)\overline{Q}(U,V)\right|_{U=V=0},
\]
and its tensor product over the archimedean places.
This normalization agrees with the pairing used in \cite[Section~6]{Takeda2025DifferentialOperators}.
For the paired weights \(\rho_m\) and \(\rho_m'\),
we choose these inner products compatibly with the conjugate-linear identification obtained by complex conjugation and interchange of the two \(\GL_m(\bbC)\)-factors.
Whenever the integral converges,
for holomorphic automorphic forms \(F_1,F_2\) on \(G_m\) of the same weight \(\rho\),
put
\[
  (F_1,F_2)_{G_m} =\int_{G_m(F)\backslash G_m(\bbA_F)} \langle F_1(g),F_2(g)\rangle_\rho \,d\bar g_m.
\]
In particular,
the integral converges absolutely if one of the two forms is cuspidal.

Put
\[
  \vartheta_m = \begin{pmatrix} 0 & I_m \\ I_m & 0 \end{pmatrix}.
\]
For \(g\in G_m(\bbA_F)\),
put
\[
  g^\natural = \vartheta_m g\vartheta_m.
\]
For a Hecke character \(\zeta\) of \(E\) whose restriction to \(\det G_m(\bbA_F)\) is unitary,
define
\[
  \varphi^{\dagger,\zeta}(g) = \zeta(\det g)\,\overline{\varphi(g^\natural)}.
\]
For \(\zeta=\chi^{-1}\),
this gives conjugate-linear maps
\[
  \calA_m(\rho_m,\frakN) \longrightarrow \calA_m(\rho_m',\frakN),
\]
\[
  \calS_m(\rho_m,\frakN) \longrightarrow \calS_m(\rho_m',\frakN).
\]
Since \(|\chi(\det g)|=1\) and \(g\mapsto g^\natural\) preserves the Petersson measure,
the map \(\dagger,\chi^{-1}\) is antiunitary for the compatible inner products fixed above:
\[
  (\varphi_1^{\dagger,\chi^{-1}},\varphi_2^{\dagger,\chi^{-1}})_{G_m} = (\varphi_2,\varphi_1)_{G_m}.
\]
\subsection{Differential operators}\label{subsec:differential-operators}

We recall the differential operator constructed in \cite{Takeda2025DifferentialOperators,Takeda2026PullbackFormula} for the pullback from \(G_{n+r}\) to \(G_n\times G_r\).

Let
\begin{align*}
  \iota_{n,r}:G_n\times G_r & \longrightarrow G_{n+r} \\ \left( \begin{pmatrix}A&B\\ C&D\end{pmatrix}, \begin{pmatrix}a&b\\ c&d\end{pmatrix} \right) & \longmapsto \begin{pmatrix} A & 0 & B & 0 \\ 0 & a & 0 & b \\ C & 0 & D & 0 \\ 0 & c & 0 & d \end{pmatrix}
\end{align*}
be the standard diagonal embedding.
For \(v\in\bfa\),
under the standard identification \(\frakp_{n+r,v}^{+}\simeq\M_{n+r}(\bbC)\),
let
\[
  \pi_v^+ = (\pi^+_{v,ij})_{1\le i,j\le n+r}
\]
denote the corresponding matrix of right-invariant differential operators associated with the holomorphic tangent directions.
Since \(\frakp_{n+r,v}^{+}\) is abelian,
the operators \(\pi^+_{v,ij}\) commute.
Thus a polynomial
\[
  P: \M_{n+r}(\bbC) \longrightarrow V_{\sigma_{n,v}}\otimes V_{\sigma'_{r,v}}
\]
determines a vector-valued differential operator \(P(\pi_v^+)\).

We require \(P(\pi_v^+)\) to satisfy the following covariance relation:
\[
  P(\pi_v^+)F(\iota_{n,r}(g_1k_1,g_2k_2)) = ( \rho_{n,v}(k_1)^{-1} \otimes \rho'_{r,v}(k_2)^{-1} ) P(\pi_v^+)F(\iota_{n,r}(g_1,g_2))
\]
for every scalar-valued holomorphic function \(F\) on \(G_{n+r}(F_v)\) of \(K_{n+r,v}\)-type \({\det}^{\kappa_v}\boxtimes{\det}^{\nu_v}\),
and all \(g_1\in G_n(F_v)\),
\(g_2\in G_r(F_v)\),
\(k_1\in K_{n,v}\),
and \(k_2\in K_{r,v}\).

The following polynomial criterion characterizes this covariance relation.
Let \(X_1,Y_1\in\M_{n,\mu}(\bbC)\) and \(X_2,Y_2\in\M_{r,\mu}(\bbC)\).
For \(P\) as above,
put
\[
  \widetilde P(X_1,X_2,Y_1,Y_2) = P\!\left( \begin{pmatrix} X_1{}^tY_1 & X_1{}^tY_2 \\ X_2{}^tY_1 & X_2{}^tY_2 \end{pmatrix} \right).
\]
A polynomial \(Q(X,Y)\) in the entries of two \(a\times\mu\) matrices is called \emph{pluriharmonic} if
\[
  \sum_{s=1}^{\mu} \frac{\partial^2Q} {\partial X_{is}\partial Y_{js}} =0
\]
for every \(1\le i,j\le a\).

See \cite[Proposition~3.2]{Takeda2025DifferentialOperators}.

\begin{prop}\label{prop:differential-criterion}
  Fix \(v\in\bfa\), and suppose that \(\ell(\bfk_v)\le r\), \(\ell(\bfl_v)\le r\), and \(\ell(\bfk_v)+\ell(\bfl_v)\le\mu\).
  Let \(P:\M_{n+r}(\bbC)\rightarrow V_{\sigma_{n,v}}\otimes V_{\sigma'_{r,v}}\) be a polynomial.
  For \(X_1,Y_1\in\M_{n,\mu}(\bbC)\), \(X_2,Y_2\in\M_{r,\mu}(\bbC)\),
  put
  \[
    \widetilde P(X_1,X_2,Y_1,Y_2)
    =
    P\!\left(
    \begin{pmatrix}
        X_1{}^tY_1 & X_1{}^tY_2 \\
        X_2{}^tY_1 & X_2{}^tY_2
      \end{pmatrix}
    \right).
  \]
  Then the covariance relation above holds if and only if the following two conditions hold:
  \begin{enumerate}
    \item
          \(\widetilde P\) is pluriharmonic in each of the pairs
          \((X_1,Y_1)\) and \((X_2,Y_2)\).

    \item
          For all
          \(A_1,B_1\in\GL_n(\bbC)\) and
          \(A_2,B_2\in\GL_r(\bbC)\),
          \[
            P\!\left(
            \begin{pmatrix}
                A_1 & 0   \\
                0   & A_2
              \end{pmatrix}
            T
            \begin{pmatrix}
                {}^tB_1 & 0       \\
                0       & {}^tB_2
              \end{pmatrix}
            \right)
            =
            \bigl(
            \sigma_{n,v}(A_1,B_1)
            \otimes
            \sigma'_{r,v}(A_2,B_2)
            \bigr)P(T).
          \]
  \end{enumerate}
\end{prop}

For Siegel modular forms,
the corresponding covariance problem is Case~(I) in~\cite{Ibukiyama1999DifferentialOperators};
see also~\cite{Ibukiyama2020GenericDifferentialOperators} for the general construction by invariant pluriharmonic polynomials.
The use of such differential operators in pullback formulas goes back to B\"ocherer--Satoh--Yamazaki~\cite{BoechererSatohYamazaki1992}.
For the present Hermitian setting,
we use the normalized choice in \cite[Theorem~5.13]{Takeda2025DifferentialOperators};
uniqueness up to scalar follows from \cite[Proposition~5.4]{Takeda2025DifferentialOperators}.
For each \(v\in\bfa\),
let
\[
  P_{\mu;\bfk_v,\bfl_v}^{(n,r)}: \M_{n+r}(\bbC) \longrightarrow V_{\sigma_{n,v}}\otimes V_{\sigma'_{r,v}}
\]
be the resulting normalized polynomial satisfying the two conditions above.

For
\(R=(R_v)_{v\in\bfa}\in\prod_{v\in\bfa}\M_{n+r}(\bbC)\),
put
\[
  P_{\mu;\bfk,\bfl}^{(n,r)}(R) = \bigotimes_{v\in\bfa} P_{\mu;\bfk_v,\bfl_v}^{(n,r)}(R_v),
\]
and define
\[
  \bbD_{\bfk,\bfl} = \bigotimes_{v\in\bfa} P_{\mu;\bfk_v,\bfl_v}^{(n,r)}(\pi_v^+).
\]
Then
\[
  (2\pi\sqrt{-1})^{-d_{\bfk,\bfl}} \bbD_{\bfk,\bfl} \exp( 2\pi\sqrt{-1} \tr_\infty(RZ) ) = P_{\mu;\bfk,\bfl}^{(n,r)}(R) \exp( 2\pi\sqrt{-1} \tr_\infty(RZ) ),
\]
where we put \(d_{\bfk,\bfl}=\sum_{v\in\bfa}(|\bfk_v|+|\bfl_v|)\).

\subsection{Siegel Eisenstein series}\label{subsec:siegel-eisenstein-series}

We define the Siegel Eisenstein series with the normalization of \cite[Section~4.2]{Takeda2026Congruences} and \cite[\S~16.5]{Shimura2000Arithmeticity}.

We use the convention \(G_0=\{1\}\).
For \(0\le q<m\),
let \(U_{m,q}\) be the unipotent radical whose \(F\)-points are
\[
  U_{m,q}(F) = \left\{ \begin{pmatrix} I_{m-q} & * & * & * \\ 0 & I_q & * & 0 \\ 0 & 0 & I_{m-q} & 0 \\ 0 & * & * & I_q \end{pmatrix} \in G_m(F) \right\}.
\]
For \(0\le q<m\), define the standard embeddings
\[
  t_{m,q}:\GL_{m-q}(E)\longrightarrow G_m(F)
\]
and
\[
  s_{m,q}:G_q(F)\longrightarrow G_m(F)
\]
by
\[
  t_{m,q}(A) = \begin{pmatrix} A & 0 & 0 & 0 \\ 0 & I_q & 0 & 0 \\ 0 & 0 & (A^*)^{-1} & 0 \\ 0 & 0 & 0 & I_q \end{pmatrix}
\]
and
\[
  s_{m,q}\!\left( \begin{pmatrix} a & b \\ c & d \end{pmatrix} \right) = \begin{pmatrix} I_{m-q} & 0 & 0 & 0 \\ 0 & a & 0 & b \\ 0 & 0 & I_{m-q} & 0 \\ 0 & c & 0 & d \end{pmatrix}.
\]
Let \(P_{m,q}\) be the standard parabolic subgroup with
\[
  P_{m,q}(F) = s_{m,q}(G_q(F)) t_{m,q}(\GL_{m-q}(E)) U_{m,q}(F).
\]
We use the same notation for the corresponding maps and subgroups over \(F_v\) and \(\bbA_F\).
The parabolic \(P_{m,0}\) is the Siegel parabolic.
For \(0<r<n\),
\(P_{n,r}\) is the Klingen parabolic.
For \(q=m\),
we put \(P_{m,m}=G_m\),
\(U_{m,m}=\{1\}\),
let \(t_{m,m}\) be the trivial map from \(\GL_0(E)=\{1\}\),
and let \(s_{m,m}\) be the identity map on \(G_m\).
Finally,
put
\[
  w_m=J_m = \begin{pmatrix} 0 & -I_m \\ I_m & 0 \end{pmatrix} \in G_m(F).
\]

Let \(\chi\) be the Hecke character fixed above.
For each place \(v\) of \(F\),
we regard its local component as a character of \(E_v^\times=\prod_{w\mid v}E_w^\times\) by
\[
  \chi_v((x_w)_{w\mid v}) = \prod_{w\mid v}\chi_w(x_w).
\]

For \(v\in\bfa\) and \(g_v =
\begin{pmatrix}
  A_v & B_v \\
  C_v & D_v
\end{pmatrix}
\in G_m(F_v)\simeq\U(m,m)\),
put
\[
  \delta_m(g_v) = \det(C_v\sqrt{-1}I_m+D_v).
\]

For each place \(v\),
let \(I_{m,v}(s,\chi_v)\) be the Siegel degenerate principal series of \(G_m(F_v)\) induced from \(P_{m,0}(F_v)\),
in the normalization of \cite{Shimura1997EulerProducts,Shimura2000Arithmeticity}.
Thus \(I_{m,v}(s,\chi_v)\) is realized as the space of smooth functions \(\Phi(\,\cdot\,,s):G_m(F_v)\rightarrow\bbC\) satisfying
\[
  \Phi(t_{m,0}(A)ug,s) = |\det(A^*A)|_v^s \chi_v(\det A^*) \Phi(g,s)
\]
for \(A\in\GL_m(E_v)\),
\(u\in U_{m,0}(F_v)\),
and \(g\in G_m(F_v)\).

For \(v\in\bfh\) and a nonzero ideal \(\fraka_v\subset\calO_{F_v}\),
let \(\varepsilon_{m,v}^{\fraka_v}(\,\cdot\,,s;\chi_v)\in I_{m,v}(s,\chi_v)\) be the right \(K_{m,v}(\fraka_v)\)-invariant section supported on \(P_{m,0}(F_v)K_{m,v}(\fraka_v)\) and normalized by
\[
  \varepsilon_{m,v}^{\fraka_v}(1,s;\chi_v)=1.
\]

For \(v\in\bfa\),
we take the holomorphic section
\[
  \varepsilon_{m,v}^{\hol}(g_v,s;\chi_v) = \det(g_v)^{\nu_v} |\delta_m(g_v)|^{\kappa_v+\nu_v-2s} \delta_m(g_v)^{-\kappa_v-\nu_v}.
\]
Put
\[
  \varepsilon_{m,\infty}^{\hol}(g_\infty,s;\chi_{\bfa}) = \prod_{v\in\bfa} \varepsilon_{m,v}^{\hol}(g_v,s;\chi_v).
\]

Fix a nonzero integral ideal \(\frakN\subset\calO_F\) prime to \(2\frakD_{E/F}\).

For \(m\ge1\),
define the global section
\[
  \varepsilon_{m,\kappa,\nu,\frakN}(g,s;\chi) = \prod_{v\in\bfh} \varepsilon_{m,v}^{\frakN_v}(g_v,s;\chi_v) \prod_{v\in\bfa} \varepsilon_{m,v}^{\hol}(g_v,s;\chi_v).
\]
Define the corresponding Siegel Eisenstein series by
\[
  E_{m,\kappa,\nu,\frakN}(g,s;\chi) = \sum_{\gamma\in P_{m,0}(F)\backslash G_m(F)} \varepsilon_{m,\kappa,\nu,\frakN}(\gamma g,s;\chi).
\]
At \(s=\mu/2\),
write
\[
  E_{m,\kappa,\nu,\frakN}(g;\chi) = E_{m,\kappa,\nu,\frakN}(g,\mu/2;\chi).
\]
The series converges absolutely and locally uniformly for \(\Re(s)>m\).
We use its continuation to \(s=\mu/2\) when this point lies outside the convergence region.
By \cite[Proposition~17.7]{Shimura2000Arithmeticity},
the specialization attached to the holomorphic section is holomorphic when \(\mu\ge m\),
except possibly when \(\mu=m+1\), \(F=\bbQ\), and \(\chi_0=\chi_{E/F}^{m+1}\).
The same definition and regularity statement will be used for the Siegel Eisenstein series appearing below when one of the two scalar weights is zero.
The exceptional case is excluded in Theorem~\ref{thm:global-pullback};
all specializations used below are therefore regular by this proposition.
\begin{rem}\label{rem:siegel-fourier-coefficients}
  The Fourier coefficients of \(E_{m,\kappa,\nu,\frakN}\) factor into local Fourier integrals;
  see \cite[\S\S18--19]{Shimura1997EulerProducts} and \cite[\S2.2]{Eischen2015EisensteinMeasure}.
  At almost all finite places these are given by the local Hermitian Siegel series,
  with explicit formulas in \cite{Hironaka1998LocalZeta,Hironaka1999Spherical}.
  At the remaining places they depend on the chosen sections.
  We do not need a uniform formula for these factors.
\end{rem}
\section{Pullback formula}\label{sec:pullback}

We establish the pullback formula for principal congruence subgroups,
including an explicit determination of the proportionality constant.

From now on,
assume that \(f\) generates an irreducible cuspidal automorphic representation \(\pi_f=\bigotimes_v'\pi_{f,v}\).
Write \(f^\dagger=f^{\dagger,\chi^{-1}}\),
with the normalization fixed above.
We express the pullback pairing in \cite[\S~5.1]{Takeda2026PullbackFormula} using this conjugate-linear convention,
with the kernel in the first variable of the Petersson product.

For a finite place \(v\nmid\frakN\),
put
\[
  L_v(s,f,\St)=L_v(s,\pi_{f,v},\St).
\]
Here \(\St\) denotes the standard \(L\)-function,
with the local normalization used in \cite[\S~5.3.1]{Takeda2026PullbackFormula}.
Define
\[
  L^{(\frakN)}(s,f,\St)=\prod_{\substack{v\in\bfh\\v\nmid\frakN}}L_v(s,f,\St).
\]
For a Hecke character \(\lambda=\bigotimes_v\lambda_v\) of \(F\),
write
\[
  L_{F,v}(s,\lambda_v)=
  \begin{cases}
    (1-\lambda_v(\varpi_v)q_v^{-s})^{-1}, & \lambda_v\text{ unramified}, \\
    1,                                    & \lambda_v\text{ ramified},
  \end{cases}
\]
for \(v\in\bfh\),
and put
\[
  L^{(\frakN)}(s,\lambda)=\prod_{\substack{v\in\bfh\\v\nmid\frakN}}L_{F,v}(s,\lambda_v).
\]
These products are initially defined in a right half-plane and elsewhere by meromorphic continuation.

For a place \(v\) and \(\varphi_v\in\pi_{f,v}\),
put
\[
  \varphi_v^\dagger(g)
  =\varphi_v^{\dagger,\chi_v^{-1}}(g)
  =\chi_v(\det g)^{-1}\,\overline{\varphi_v(g^\natural)}.
\]

By \cite[Fact~19, Proposition~20]{Takeda2026PullbackFormula},
after unfolding,
only the open orbit contributes,
while all other orbits vanish by cuspidality.

\subsection{Local pullback integrals}\label{subsec:local-pullback}

The spherical and archimedean integrals are evaluated in \cite[\S\S~5.3.1 and~5.3.3]{Takeda2026PullbackFormula}.
We recall these formulas and compute the integrals at the primes dividing \(\frakN\).

Let \(v\in\bfh\),
let \(\fraka_v\subset\calO_{F_v}\) be a nonzero ideal,
and let \(\varphi_v\in\pi_{f,v}^{K_{r,v}(\fraka_v)}\).
Define \(\Psi_{n,r,v}^{\fraka_v}(\varphi_v;\,\cdot\,,s;\chi_v)\) to be the local Klingen section supported on \(P_{n,r}(F_v)K_{n,v}(\fraka_v)\) and satisfying
\[
  \Psi_{n,r,v}^{\fraka_v}( \varphi_v; t_{n,r}(A)u\,s_{n,r}(h)k, s;\chi_v ) = |\det(A^*A)|_v^s \chi_v(\det A^*) \varphi_v(h)
\]
for \(A\in\GL_{n-r}(E_v)\),
\(u\in U_{n,r}(F_v)\),
\(h\in G_r(F_v)\),
and \(k\in K_{n,v}(\fraka_v)\).

Put
\[
  e_r = \begin{pmatrix} 0_{n-r, r} \\ I_r \end{pmatrix} \in\M_{n,r}(\calO_F),
\]
and take
\[
  \xi_r = \begin{pmatrix} I_n & 0 & 0 & 0 \\ 0 & I_r & 0 & 0 \\ 0 & e_r & I_n & 0 \\ {}^te_r & 0 & 0 & I_r \end{pmatrix} \in G_{n+r}(F)
\]
as a representative of the open orbit for the right action of \(\iota_{n,r}(G_n\times G_r)\) on \(P_{n+r,0}\backslash G_{n+r}\).

See \cite[Proposition~24]{Takeda2026PullbackFormula}.

\begin{prop}\label{prop:unramified-local-pullback}
  Let \(v\) be a finite place at which \(\pi_{f,v}\) and the local Siegel section are unramified,
  and let \(\varphi_v\in\pi_{f,v}^{K_{r,v}}\).
  Then,
  for every \(g_1\in G_n(F_v)\),
  \begin{align*}
     & \int_{G_r(F_v)} \varepsilon_{n+r,v}^{\calO_{F_v}} \left( \xi_r\iota_{n,r}(g_1,g_2), \frac{\mu}{2};\chi_v \right) \overline{\varphi_v^\dagger(g_2)} \,dg_2 \\ & \quad= L_v\!\left(\frac{\mu-2r+1}{2},\pi_{f,v}, \St\right) \prod_{a=0}^{2r-1}L_{F,v}(\mu-a,\chi_{0,v}\chi_{E/F,v}^{a})^{-1} \Psi_{n,r,v}^{\calO_{F_v}} \left( \varphi_v; g_1, \frac{\mu}{2};\chi_v \right).
  \end{align*}
\end{prop}

Let \(m\in\{r,n\}\) and \(v\in\bfa\).
For \(x=\begin{pmatrix}
  A & B \\
  C & D
\end{pmatrix}\in G_m(F_v)\),
put
\[
  \calM_m(x) = ( C\sqrt{-1}I_m+D, \overline C\sqrt{-1}I_m+\overline D ).
\]
Let \(\varphi_v\in\pi_{f,v}\) belong to the holomorphic \(K_{r,v}\)-type corresponding to \(\rho_{r,v}\).
Use the bideterminant realizations of the polynomial representations \(\rho_{m,\bfk_v}\) and \(\rho_{m,\bfl_v}\),
with matrix variables \(U\) and \(V\) whose columns are indexed by \(1,\ldots,m\).
The substitutions
\[
  u_{ai}\longmapsto u_{a,n-r+i}, \qquad v_{bi}\longmapsto v_{b,n-r+i} \qquad (1\le i\le r)
\]
send each bideterminant to the corresponding bideterminant in the last \(r\) columns.
Via these substitutions,
we regard \(V_{\sigma_{r,v}}\) as a subspace of \(V_{\sigma_{n,v}}\),
and likewise \(V_{\rho_{r,v}}\) as a subspace of \(V_{\rho_{n,v}}\) under the vector-space identifications fixed above.
For \(g=t_{n,r}(A)u\,s_{n,r}(h)k\) with \(A\in\GL_{n-r}(E_v)\),
\(u\in U_{n,r}(F_v)\),
\(h\in G_r(F_v)\),
and \(k\in K_{n,v}\),
define
\begin{align*}
  \Psi_{n,r,v}^{\hol}(\varphi_v;g,s;\chi_v) ={} & |\delta_n(g)\delta_r(h)^{-1}|^{\kappa_v+\nu_v-2s} \rho_{n,v}(\calM_n(g))^{-1} \rho_{r,v}(\calM_r(h))\varphi_v(h).
\end{align*}
For a pure tensor \(\varphi_\infty=\bigotimes_{v\in\bfa}\varphi_v\),
put
\[
  \Psi_{n,r,\infty}^{\hol}(\varphi_\infty;g,s;\chi_{\bfa}) = \bigotimes_{v\in\bfa} \Psi_{n,r,v}^{\hol}(\varphi_v;g_v,s;\chi_v),
\]
and extend this definition linearly in \(\varphi_\infty\).

See \cite[Proposition~27]{Takeda2026PullbackFormula},
\cite[Proposition~4.9]{Takeda2026Congruences},
\cite[Theorem~6.6]{Takeda2025DifferentialOperators}.

\begin{prop}\label{prop:arche-local-pullback}
  Put \(\pi_{f,\infty}=\bigotimes_{v\in\bfa}\pi_{f,v}\).
  For \(\varphi_\infty\in\pi_{f,\infty}\) in the holomorphic \(K_{r,\infty}\)-type corresponding to \(\rho_r\),
  define \(\varphi_\infty^\dagger\) componentwise by the same involution.
  Then,
  for every \(g_1\in G_n(\bbA_{F,\infty})\),
  \[
    \int_{G_r(\bbA_{F,\infty})} \left\langle \bbD_{\bfk,\bfl} \varepsilon_{n+r,\infty}^{\hol} \left( \xi_r\iota_{n,r}(g_1,g_2), \frac{\mu}{2};\chi_{\bfa} \right), \varphi_\infty^\dagger(g_2) \right\rangle \,dg_2 = c_\infty \Psi_{n,r,\infty}^{\hol} \left( \varphi_\infty; g_1, \frac{\mu}{2};\chi_{\bfa} \right),
  \]
  where
  \begin{align*}
    c_\infty ={} & 2^{-d_F(2r^2+(r-2)\mu)+d_{\bfk,\bfl}} \pi^{d_Fr^2} \prod_{v\in\bfa} \prod_{i=1}^{r} \prod_{t=0}^{r-1} (k_{v,i}+l_{v,i}+\mu-i-t)^{-1} \\ & \times \prod_{v\in\bfa} \frac{ \prod_{i=1}^{\ell(\bfk_v)} (k_{v,i}+\ell(\bfk_v)-i)! }{ \prod_{1\le i<j\le\ell(\bfk_v)} (k_{v,i}-k_{v,j}+j-i) } \prod_{v\in\bfa} \frac{ \prod_{i=1}^{\ell(\bfl_v)} (l_{v,i}+\ell(\bfl_v)-i)! }{ \prod_{1\le i<j\le\ell(\bfl_v)} (l_{v,i}-l_{v,j}+j-i) }.
  \end{align*}
\end{prop}

It remains to compute the local integral at \(v\mid\frakN\).
Fix such a place \(v\) and use the section \(\varepsilon_{m,v}^{\frakN_v}(\,\cdot\,,s;\chi_v)\) of Subsection~\ref{subsec:siegel-eisenstein-series}.

\begin{lem}\label{lem:level-big-cell-criterion}
  Let \(m\ge1\),
  and let \(x=\begin{pmatrix}
    A & B \\
    C & D
  \end{pmatrix} \in G_m(F_v)\).
  Then
  \[
    x\in P_{m,0}(F_v)K_{m,v}(\frakN_v)
  \]
  if and only if \(D\) is invertible and \(D^{-1}C\in\M_m(\frakN_v\calO_{E,v})\).
\end{lem}
\begin{proof}
  If \(x=pk\) with \(p\in P_{m,0}(F_v)\) and \(k=
  \begin{pmatrix}
    a & b \\
    c & d
  \end{pmatrix}\in K_{m,v}(\frakN_v)\),
  then \(d\equiv I_m\pmod{\frakN_v}\),
  so \(d\) is invertible.
  Since the lower blocks of \(x\) are obtained from \((c,d)\) by the same invertible Levi factor,
  \[
    D^{-1}C=d^{-1}c\in\M_m(\frakN_v\calO_{E,v}).
  \]

  Conversely,
  suppose that \(D\) is invertible and
  \[
    U=D^{-1}C\in\M_m(\frakN_v\calO_{E,v}).
  \]
  From \(DC^*=CD^*\),
  we have \(U^*=U\),
  and hence
  \[
    \overline n(-U) = \begin{pmatrix} I_m & 0 \\ -U & I_m \end{pmatrix} \in K_{m,v}(\frakN_v).
  \]
  Since the lower-left block of \(x\overline n(-U)\) is zero,
  \[
    x\overline n(-U)\in P_{m,0}(F_v),
  \]
  proving the converse.
\end{proof}

Let \(\varphi_v\in\pi_{f,v}^{K_{r,v}(\frakN_v)}\),
and use the section \(\Psi_{n,r,v}^{\frakN_v}(\varphi_v;\,\cdot\,,s;\chi_v)\) defined above.

\begin{prop}\label{prop:principal-local-pullback}
  We have
  \[
    \int_{G_r(F_v)} \varepsilon_{n+r,v}^{\frakN_v}
    \bigl( \xi_r \iota_{n,r}(g_1,g_2) \xi_r^{-1}, s;\chi_v \bigr)
    \overline{\varphi_v^\dagger(g_2)} \,dg_2
    =
    \vol(K_{r,v}(\frakN_v))
    \Psi_{n,r,v}^{\frakN_v}(\varphi_v;g_1,s;\chi_v).
  \]
\end{prop}
\begin{proof}
  Put
  \[
    Z_v(g_1,s)
    =
    \int_{G_r(F_v)} \varepsilon_{n+r,v}^{\frakN_v}
    \bigl( \xi_r \iota_{n,r}(g_1,g_2) \xi_r^{-1}, s;\chi_v \bigr)
    \overline{\varphi_v^\dagger(g_2)} \,dg_2.
  \]

  We first determine the support and value of the section at \(\xi_r\iota_{n,r}(s_{n,r}(g),I_{2r})\xi_r^{-1}\).
  Write
  \[
    g=\begin{pmatrix}a&b\\c&d\end{pmatrix}\in G_r(F_v).
  \]
  With respect to the decomposition \(n=(n-r)+r\),
  the complementary blocks of size \(n-r\) are identity or zero blocks,
  and the nontrivial \(4r\times4r\) block of \(\xi_r \iota_{n,r}(s_{n,r}(g),I_{2r}) \xi_r^{-1}\)
  is
  \[
    \begin{pmatrix} a & -b & b & 0 \\ 0 & I_r & 0 & 0 \\ c & I_r-d & d & 0 \\ a-I_r & -b & b & I_r \end{pmatrix}.
  \]
  Its lower blocks are
  \[
    C_0= \begin{pmatrix} c & I_r-d \\ a-I_r & -b \end{pmatrix}, \qquad D_0= \begin{pmatrix} d & 0 \\ b & I_r \end{pmatrix}.
  \]
  By Lemma~\ref{lem:level-big-cell-criterion},
  the section is supported precisely when \(d\) is invertible and
  \[
    D_0^{-1}C_0 = \begin{pmatrix} d^{-1}c & d^{-1}-I_r \\ a-I_r-bd^{-1}c & -bd^{-1} \end{pmatrix} \in\M_{2r}(\frakN_v\calO_{E,v}).
  \]
  This is equivalent to \(a\equiv d\equiv I_r, b\equiv c\equiv0 \pmod{\frakN_v}\),
  hence to \(g\in K_{r,v}(\frakN_v)\).
  If \(g\in K_{r,v}(\frakN_v)\),
  then \(\xi_r\iota_{n,r}(s_{n,r}(g),I_{2r})\xi_r^{-1}\in K_{n+r,v}(\frakN_v)\).
  By the right \(K_{n+r,v}(\frakN_v)\)-invariance of the section and its normalization at the identity,
  we obtain
  \[
    \varepsilon_{n+r,v}^{\frakN_v} \bigl(\xi_r\iota_{n,r}(s_{n,r}(g),I_{2r})\xi_r^{-1},s;\chi_v\bigr) =\varepsilon_{n+r,v}^{\frakN_v}(1,s;\chi_v) =1.
  \]

  Thus
  \[
    \varepsilon_{n+r,v}^{\frakN_v}
    \bigl(\xi_r\iota_{n,r}(s_{n,r}(g),I_{2r})\xi_r^{-1},s;\chi_v\bigr)
    =\mathbbm1_{K_{r,v}(\frakN_v)}(g).
  \]

  By the calculation of the stabilizer of the open orbit in \cite[Fact~19(2)]{Takeda2026PullbackFormula} and the description of the support above,
  \(Z_v(g_1,s)\) is supported on \(P_{n,r}(F_v)K_{n,v}(\frakN_v)\).
  It has the same \(P_{n,r}(F_v)\)-equivariance and right \(K_{n,v}(\frakN_v)\)-invariance as \(\Psi_{n,r,v}^{\frakN_v}(\varphi_v;g_1,s;\chi_v)\).
  It therefore suffices to take \(g_1=s_{n,r}(h)\).
  The change of variables and the transformation law of the section yield
  \begin{align*}
    Z_v(s_{n,r}(h),s)
     & =\int_{G_r(F_v)}
    \varepsilon_{n+r,v}^{\frakN_v}
    \bigl(\xi_r\iota_{n,r}(s_{n,r}(h),(hg^{-1})^\natural)\xi_r^{-1},s;\chi_v\bigr)
    \overline{\varphi_v^\dagger((hg^{-1})^\natural)}\,dg   \\
     & =\int_{G_r(F_v)}
    \mathbbm1_{K_{r,v}(\frakN_v)}(g)\varphi_v(hg^{-1})\,dg \\
     & =\int_{K_{r,v}(\frakN_v)}\varphi_v(hg^{-1})\,dg     \\
     & =\vol(K_{r,v}(\frakN_v))
    \Psi_{n,r,v}^{\frakN_v}(\varphi_v;s_{n,r}(h),s;\chi_v).
  \end{align*}
\end{proof}

The volume factor is also explicit.
Since \(v\nmid\frakD_{E/F}\),
the reductive \(\calO_{F_v}\)-group scheme defining the hyperspecial subgroup \(K_{r,v}\) is smooth,
and our normalization \(\vol(K_{r,v})=1\) gives
\[
  \vol(K_{r,v}(\frakp_v^c)) = \#G_r(\calO_{F_v}/\frakp_v^c)^{-1}.
\]
Using the orders of \(\GL_{2r}(\bbF_{q_v})\) in the split case and \(\U_{2r}(\bbF_{q_v})\) in the inert case,
we obtain,
for \(c\ge1\),
\[
  \vol(K_{r,v}(\frakp_v^c)) = \begin{cases} \displaystyle q_v^{-4r^2(c-1)-r(2r-1)} \prod_{j=1}^{2r}(q_v^j-1)^{-1}, & v\text{ split in }E, \\[8pt] \displaystyle q_v^{-4r^2(c-1)-r(2r-1)} \prod_{j=1}^{2r} (q_v^j-(-1)^j)^{-1}, & v\text{ inert in }E/F. \end{cases}
\]

\subsection{Global pullback formula}\label{subsec:global-pullback}

For a pure tensor \(f=\bigotimes_v f_v\),
define the global Klingen section by
\[
  \Psi_{n,r,\frakN}(f;g,s;\chi)
  =\left(\prod_{v\in\bfh}\Psi_{n,r,v}^{\frakN_v}(f_v;g_v,s;\chi_v)\right)
  \otimes\left(\bigotimes_{v\in\bfa}\Psi_{n,r,v}^{\hol}(f_v;g_v,s;\chi_v)\right),
\]
using spherical vectors normalized by \(f_v(1)=1\) at almost all finite places.
Extend this definition linearly to general \(f\).
Put
\[
  [f]_{\kappa,\nu,\frakN}^n(g,s;\chi) = \sum_{\gamma\in P_{n,r}(F)\backslash G_n(F)} \Psi_{n,r,\frakN}(f;\gamma g,s;\chi).
\]
so that \([f]_{\kappa,\nu,\frakN}^{r}=f\) when \(n=r\).
At the holomorphic point,
write
\[
  [f]_{\kappa,\nu,\frakN}^n(g;\chi) = [f]_{\kappa,\nu,\frakN}^n(g,\mu/2;\chi).
\]

Let \(\theta_{n,r}=\theta_{n,r}^{(\frakN)}\in G_{n+r}(\bbA_{F,\bfh})\) be the element whose \(v\)-component is
\[
  \theta_{n,r,v} = \begin{cases} \xi_r, & v\mid\frakN, \\ I_{2(n+r)}, & v\nmid\frakN. \end{cases}
\]
Put
\[
  \bbL^{(\frakN)}(m,\mu;\chi) = \prod_{a=0}^{m-1} L^{(\frakN)}( 1-\mu+a, \chi_0^{-1}\chi_{E/F}^{a} ).
\]
Define the degree-\(m\) arithmetic normalization by
\[
  \widetilde E_{m,\kappa,\nu,\frakN}(g;\chi)
  =
  \bbL^{(\frakN)}(m,\mu;\chi)
  E_{m,\kappa,\nu,\frakN}(g;\chi).
\]
For the pullback kernel, however, we use the degree-\(n\) normalization
\[
  \widetilde E_{n+r,\kappa,\nu,\frakN}^{[n],\theta}(g;\chi)
  =
  \bbL^{(\frakN)}(n,\mu;\chi)
  E_{n+r,\kappa,\nu,\frakN}(g\theta_{n,r}^{-1};\chi).
\]
Put
\[
  \calE_{\frakN}(g,h)
  =
  (2\pi\sqrt{-1})^{-d_{\bfk,\bfl}}
  \bbD_{\bfk,\bfl}
  \widetilde E_{n+r,\kappa,\nu,\frakN}^{[n],\theta}
  (\iota_{n,r}(g,h);\chi),
\]
and
\begin{align}
  C_{n+r,\kappa,\nu}^{(\frakN)}(f;\chi)
  ={} &
  (2\pi\sqrt{-1})^{-d_{\bfk,\bfl}}
  \bbL^{(\frakN)}(n,\mu;\chi)c_\infty
  \notag       \\
      & \times
  L^{(\frakN)}\left(
  \frac{\mu-2r+1}{2},f,\St
  \right)
  \prod_{a=0}^{2r-1}
  L^{(\frakN)}
  \left(
  \mu-a,\chi_0\chi_{E/F}^{a}
  \right)^{-1}
  \notag       \\
      & \times
  \prod_{v\mid\frakN}
  \vol(K_{r,v}(\frakN_v)).
  \label{eq:C-explicit-formula}
\end{align}
Suppose \(\mu\ge n+r\).
For \(0\le a\le2r-1\),
\[
  \mu-a\ge n-r+1\ge2,
\]
so the Hecke \(L\)-factors in the denominator are finite and nonzero.
Moreover,
\[
  \frac{\mu-2r+1}{2}
  \ge
  \frac{n-r+1}{2}
  \ge1.
\]
This lies in the zero-free half-plane \(\Re(s)>1\), except when
\(\mu=n+r\) and \(n=r+1\), where we assume
\(L^{(\frakN)}(1,f,\St)\) to be holomorphic and nonzero.
Finally, for \(0\le a\le n-1\) put
\(k=\mu-a\).
Then \(k\ge r+1\ge2\).
The archimedean parity of
\(\chi_0^{-1}\chi_{E/F}^{a}\)
is the parity required at \(1-k\), so the functional equation identifies the corresponding full Hecke \(L\)-value with a nonzero value in the half-plane of absolute convergence.
Removing the Euler factors at \(v\mid\frakN\) multiplies this value by
\[
  1-
  \bigl(\chi_0^{-1}\chi_{E/F}^{a}\bigr)_v(\varpi_v)
  q_v^{k-1},
\]
which is nonzero because \(k\ge2\).
Thus \(\bbL^{(\frakN)}(n,\mu;\chi)\ne0\).
The explicit factor \(c_\infty\) and the local volume factors are also nonzero.
\begin{thm}\label{thm:global-pullback}
  Assume \(\mu\ge n+r\).
  If \(\mu=n+r+1\) and \(F=\bbQ\), assume \(\chi_0\ne\chi_{E/F}^{n+r+1}\).
  If \(\mu=n+r\) and \(n=r+1\), assume that \(L^{(\frakN)}(1,f,\St)\) is holomorphic and nonzero.
  Then the Klingen--Eisenstein series \([f]_{\kappa,\nu,\frakN}^n(g,s;\chi)\) is regular at \(s=\mu/2\),
  and
  \[
    ( \calE_{\frakN}(g,*), f^\dagger )_{G_r}
    =
    C_{n+r,\kappa,\nu}^{(\frakN)}(f;\chi)
    [f]_{\kappa,\nu,\frakN}^n(g;\chi).
  \]
\end{thm}
\begin{proof}
  For \(\Re(s)>n+r\),
  unfold the pullback integral as in \cite[\S\S~5.2--5.3]{Takeda2026PullbackFormula}.
  Cuspidality eliminates all but the open orbit,
  and the local calculations there identify the remaining integral with a scalar multiple of \([f]_{\kappa,\nu,\frakN}^n(g,s;\chi)\).
  This identity extends meromorphically in \(s\).

  Put \(m=n+r\).
  Since \(\mu\ge m\),
  \cite[Proposition~17.7]{Shimura2000Arithmeticity}
  gives regularity of the degree-\(m\) Siegel Eisenstein series at \(s=\mu/2\),
  except at \(\mu=m+1\) with the exceptional character.
  This is exactly the case excluded in the statement.
  Its Petersson pairing with \(f^\dagger\) is therefore regular at \(\mu/2\),
  since \(f^\dagger\) is rapidly decreasing.

  At \(s=\mu/2\),
  Propositions~\ref{prop:unramified-local-pullback},
  \ref{prop:arche-local-pullback},
  and \ref{prop:principal-local-pullback} identify the scalar obtained from the local integrals with the nonzero value
  \(C_{n+r,\kappa,\nu}^{(\frakN)}(f;\chi)\) in \eqref{eq:C-explicit-formula}.
  Dividing the continued identity by this scalar proves regularity of \([f]_{\kappa,\nu,\frakN}^n(g,s;\chi)\) at \(\mu/2\).
  Evaluation at \(\mu/2\) gives the asserted formula.
\end{proof}

\section{Primitive Fourier coefficients}\label{sec:primitive}

We formulate the primitive Fourier coefficient as a finite M\"obius sum and apply it to the pullback kernel before unfolding.
Because the lattices \(L_{\gamma,\frakN}^{(n)}\) attached to cusp representatives over a general CM field need not be free,
the inversion is expressed in terms of projective \(\calO_E\)-lattices.

\subsection{Primitive Fourier coefficients and cusp representatives}\label{subsec:primitive-coefficients}

Put \(\widehat{\calO}_E=\calO_E\otimes_{\bbZ}\widehat{\bbZ}\) and \(W_m=E^m\oplus0\subset E^{2m}\).
For \(g\in G_m(\bbA_{F,\bfh})\),
define
\[
  L_{g,\frakN}^{(m)}
  =\frakN\cdot(W_m\cap g\widehat{\calO}_E^{2m}).
\]
This is a finitely generated torsion-free \(\calO_E\)-module,
hence a projective lattice.

Choose a set \(\calC^{(m)}(\frakN)\subset G_m(\bbA_{F,\bfh})\) of representatives for \(P_{m,0}(F)\backslash G_m(\bbA_{F,\bfh})/K_m(\frakN)\).
We choose each representative \(\gamma\) so that \(\gamma_v\in K_{m,v}\) if \(v\mid\frakN\),
while \(\gamma_v=t_{m,0}(c_{\gamma,v})\) with \(c_{\gamma,v}\in\GL_m(E_v)\) if \(v\nmid\frakN\).
Set \(c_{\gamma,v}=I_m\) for \(v\mid\frakN\) and write \(c_\gamma=(c_{\gamma,v})_{v\in\bfh}\).
The lattice attached to \(\gamma\) then satisfies
\[
  L_{\gamma,\frakN}^{(m)}\otimes\widehat{\calO}_E
  =\frakN\cdot c_\gamma\widehat{\calO}_E^{m}.
\]

For \(g\in G_m(\bbA_{F,\bfh})\),
put
\[
  \fraku_g(\frakN)
  =
  \{Y\in\Her_m(E):n_m(Y)\in G_m(F)\cap gK_m(\frakN)g^{-1}\},
\]
and
\[
  \Lambda_g(\frakN)
  =
  \{T\in\Her_m(E):\tr(TY)\in\frakd_F^{-1}\text{ for all }Y\in\fraku_g(\frakN)\}.
\]
Let \(F_0\) be a holomorphic automorphic form of weight \(\rho\),
and let \(g\in G_m(\bbA_{F,\bfh})\),
viewed as an adelic element with identity archimedean component.
For \(T\in\Her_m(E)\),
its normalized \(T\)-th Fourier coefficient on the component represented by \(g\) is defined by
\[
  A_{F_0}(g,T)
  =e^{2\pi\tr_\infty(T)}
  \int_{[\Her_m]}
  F_0(n_m(Y)g)\overline{\psi_F(\tr(TY))}\,dY,
\]
where we put \([\Her_m]=\Her_m(F)\backslash\Her_m(\bbA_F)\).
Here \(dY\) is the quotient measure induced by the fixed self-dual Haar measure.
With this normalization,
the corresponding holomorphic function \(F_{0,g}\) has the expansion
\[
  F_{0,g}(Z)
  =
  \sum_T A_{F_0}(g,T)e^{2\pi\sqrt{-1}\tr_\infty(TZ)}.
\]
At level \(K_m(\frakN)\),
we have \(A_{F_0}(g,T)=0\) unless \(T\in\Lambda_g(\frakN)\).

For \(\gamma\in\calC^{(n)}(\frakN)\) and \(T\in\Lambda_\gamma(\frakN)_{>0}\),
put
\[
  (L_{\gamma,\frakN}^{(n)})_T^\sharp
  =
  \left\{
  y\in E^n:
  \Tr_{E/F}(y^*Tx)\in\frakd_F^{-1}
  \text{ for all }x\in L_{\gamma,\frakN}^{(n)}
  \right\}.
\]

\begin{dfn}\label{def:primitive-lattice-poset}
  Let \(\calP_{\gamma,\frakN}^{(n)}(T)\) be the set of projective \(\calO_E\)-lattices \(M\subset E^n\) such that
  \begin{enumerate}
    \item \(L_{\gamma,\frakN}^{(n)}\subseteq M\subseteq(L_{\gamma,\frakN}^{(n)})_T^\sharp\);
    \item \(x^*Tx\in\frakd_F^{-1}\) for every \(x\in M\);
    \item \(M_v=L_{\gamma,\frakN,v}^{(n)}\) for every \(v\mid\frakN\).
  \end{enumerate}
\end{dfn}

The quotient \((L_{\gamma,\frakN}^{(n)})_T^\sharp/L_{\gamma,\frakN}^{(n)}\) is finite,
so \(\calP_{\gamma,\frakN}^{(n)}(T)\) is a finite poset.
Write \(\calP_\gamma(T)=\calP_{\gamma,\frakN}^{(n)}(T)\).

For a polynomial representation \((\rho,V_\rho)\) of \(K_{m,\infty}^{\bbC}\),
\(a\in\GL_m(E)\),
and \(Z,S\in\Her_m(E)\),
write
\[
  \rho[a]
  =
  \rho((a_v^*,{}^ta_v)_{v\in\bfa}),
  \qquad
  \rho_S[a,Z]
  =
  \psi_{F,\bfh}(\tr(SZ))\rho[a].
\]
Thus \(\rho_S[a,Z]\in\operatorname{End}(V_\rho)\) is the product of the Fourier character arising from translation by \(Z\) and the action of \(a\) on the coefficient representation.
The normalized Fourier coefficients satisfy
\begin{equation}\label{eq:Fourier-coefficient-covariance}
  A_{F_0}(\gamma',(b^*)^{-1}Sb^{-1})
  =
  \rho[b]^{-1}A_{F_0}(\gamma',S)
\end{equation}
whenever
\(t_{n,0}(b)\in G_n(F)\cap\gamma'K_n(\frakN)(\gamma')^{-1}\).

Fix \(M\in\calP_\gamma(T)\).
Choose \(h_M\in\GL_n(\bbA_{E,\bfh})\) such that
\[
  h_M(L_{\gamma,\frakN}^{(n)}\otimes\widehat{\calO}_E)
  =M\otimes\widehat{\calO}_E,
\]
with \(h_{M,v}=I_n\) for \(v\mid\frakN\).
By the definition of \(\calC^{(n)}(\frakN)\) and the Levi decomposition of \(P_{n,0}\),
we may write
\[
  t_{n,0}(h_M)\gamma
  =t_{n,0}(a_M^{-1})n_n(Z_M)\gamma[M]k_M,
\]
where \(\gamma[M]\in\calC^{(n)}(\frakN)\),
\(a_M\in\GL_n(E)\),
\(Z_M\in\Her_n(E)\),
and \(k_M\in K_n(\frakN)\).

These choices give
\[
  a_MM=L_{\gamma[M],\frakN}^{(n)}.
\]
Put
\[
  T_M=(a_M^*)^{-1}Ta_M^{-1}.
\]
Then \(T_M\in\Lambda_{\gamma[M]}(\frakN)_{>0}\).

For \(M=L_{\gamma,\frakN}^{(n)}\),
we take \(h_M=a_M=I_n\),
\(Z_M=0\),
\(\gamma[M]=\gamma\),
and \(k_M=1\).

\begin{lem}\label{lem:fourier-coefficient-transformation}
  Suppose that
  \[
    t_{n,0}(h)\gamma
    =
    t_{n,0}(a^{-1})n_n(Z)\eta k,
  \]
  with \(a\in\GL_n(E)\),
  \(Z\in\Her_n(E)\),
  \(\eta\in\calC^{(n)}(\frakN)\),
  and \(k\in K_n(\frakN)\).
  If \(S=(a^*)^{-1}Ta^{-1}\),
  then
  \[
    A_{F_0}(t_{n,0}(h)\gamma,T)
    =
    \rho_S[a,Z]A_{F_0}(\eta,S).
  \]
  In particular,
  \begin{equation}\label{eq:fourier-coefficient-transformation}
    A_{F_0}(t_{n,0}(h_M)\gamma,T)
    =
    \rho_{T_M}[a_M,Z_M]A_{F_0}(\gamma[M],T_M).
  \end{equation}
  For fixed \(F_0\),
  the common value in \eqref{eq:fourier-coefficient-transformation} depends only on \(\gamma\),
  \(T\),
  and \(M\).
\end{lem}

\begin{proof}
  The first identity follows from the Fourier coefficient transformation formula in \cite[Section~2.2]{Takeda2026Congruences},
  with \(\alpha=a^{-1}\) and \(X_0=a^{-1}Z(a^*)^{-1}\).
  Taking \((h,a,Z,\eta)=(h_M,a_M,Z_M,\gamma[M])\) gives \eqref{eq:fourier-coefficient-transformation}.

  For fixed \(h_M\),
  \eqref{eq:fourier-coefficient-transformation} shows that the value is independent of the parabolic decomposition.
  For another choice \(h_M'\),
  the element \(u=(h_M')^{-1}h_M\) preserves \(L_{\gamma,\frakN}^{(n)}\otimes\widehat{\calO}_E\) and satisfies \(u_v=I_n\) for \(v\mid\frakN\).
  Hence \(k=\gamma^{-1}t_{n,0}(u)\gamma\in K_n(\frakN)\),
  so
  \[
    A_{F_0}(t_{n,0}(h_M)\gamma,T)
    =A_{F_0}(t_{n,0}(h_M')\gamma k,T)
    =A_{F_0}(t_{n,0}(h_M')\gamma,T).
  \]
\end{proof}

For a finite poset \(P\),
let \(\mu_P\) denote its M\"obius function.

\begin{dfn}\label{def:primitive-coefficient}
  Let \(F_0\) be an automorphic form on \(G_n\) of weight \(\rho\) and level \(K_n(\frakN)\).
  We define its primitive Fourier coefficient at \((\gamma,T)\) by
  \begin{align*}
    A_{F_0}^{\prim}(\gamma,T)
     & =
    \sum_{M\in\calP_\gamma(T)}
    \mu_{\calP_\gamma(T)}(L_{\gamma,\frakN}^{(n)},M)
    A_{F_0}(t_{n,0}(h_M)\gamma,T) \\
     & =
    \sum_{M\in\calP_\gamma(T)}
    \mu_{\calP_\gamma(T)}(L_{\gamma,\frakN}^{(n)},M)
    \rho_{T_M}[a_M,Z_M]
    A_{F_0}(\gamma[M],T_M).
  \end{align*}
\end{dfn}

\begin{lem}\label{lem:primitive-reconstruction}
  For \(M\in\calP_\gamma(T)\), the map
  \[
    \begin{aligned}
      \{N\in\calP_\gamma(T):M\subseteq N\}
        & \xrightarrow{\sim}
      \calP_{\gamma[M]}(T_M), \\
      N & \longmapsto a_MN
    \end{aligned}
  \]
  is an isomorphism of posets.
  Moreover,
  \begin{equation}\label{eq:primitive-reconstruction}
    A_{F_0}(\gamma,T)
    =
    \sum_{M\in\calP_\gamma(T)}
    \rho_{T_M}[a_M,Z_M]
    A_{F_0}^{\prim}(\gamma[M],T_M).
  \end{equation}
\end{lem}

\begin{proof}
  Put \(L=L_{\gamma,\frakN}^{(n)}\) and \(P=\calP_\gamma(T)\).
  Since \(a_MM=L_{\gamma[M],\frakN}^{(n)}\) and every \(x\in E^n\) satisfies \((a_Mx)^*T_M(a_Mx)=x^*Tx\),
  the map \(N\mapsto a_MN\) gives the asserted poset isomorphism.
  It therefore identifies \(\mu_{\calP_{\gamma[M]}(T_M)}(a_MM,a_MN)\) with \(\mu_P(M,N)\).
  Write \(P_M=\calP_{\gamma[M]}(T_M)\),
  and choose \(h_{M,N'}\) for each lattice \(N'\in P_M\) at the cusp \(\gamma[M]\).
  Substituting the definition and then setting \(N'=a_MN\),
  we obtain
  \begin{align*}
     & \sum_{M\in P}
    \rho_{T_M}[a_M,Z_M]A_{F_0}^{\prim}(\gamma[M],T_M)               \\
     & \quad=\sum_{M\in P}\rho_{T_M}[a_M,Z_M]
    \sum_{N'\in P_M}\mu_{P_M}(a_MM,N')
    A_{F_0}(t_{n,0}(h_{M,N'})\gamma[M],T_M)                         \\
     & \quad=\sum_{M\in P}\sum_{\substack{N\in P                    \\M\subseteq N}}
    \mu_P(M,N)\rho_{T_M}[a_M,Z_M]
    A_{F_0}(t_{n,0}(h_{M,a_MN})\gamma[M],T_M)                       \\
     & \quad=\sum_{M\in P}\sum_{\substack{N\in P                    \\M\subseteq N}}
    \mu_P(M,N)A_{F_0}(t_{n,0}(h_N)\gamma,T)                         \\
     & \quad=\sum_{N\in P}
    \left(\sum_{L\subseteq M\subseteq N}\mu_P(M,N)\right)
    A_{F_0}(t_{n,0}(h_N)\gamma,T)                                   \\
     & \quad=\sum_{N\in P}\delta_{L,N}A_{F_0}(t_{n,0}(h_N)\gamma,T) \\
     & \quad=A_{F_0}(\gamma,T).
  \end{align*}
  The third equality follows by applying \eqref{eq:fourier-coefficient-transformation} first to \(M\in\calP_\gamma(T)\) and then to \(a_MN\in\calP_{\gamma[M]}(T_M)\).
  The last two equalities use the M\"obius identity and the choice \(h_L=I_n\).
\end{proof}

\subsection{Theta series and Siegel Eisenstein series}\label{subsec:theta-eisenstein}

We fix \(\gamma\in\calC^{(n)}(\frakN)\), \(T\in\Lambda_\gamma(\frakN)_{>0}\), and  \(m_0=\mu-n\),
and put \(L_\gamma=L_{\gamma,\frakN}^{(n)}\).
Choose \(\calG^{(r)}(\frakN)\subset\calC^{(r)}(\frakN)\) whose images represent \(G_r(F)\backslash G_r(\bbA_{F,\bfh})/K_r(\frakN)\).
For \(\delta\in\calG^{(r)}(\frakN)\), put \(L_\delta=L_{\delta,\frakN}^{(r)}\), \(\Xi_{\gamma,\delta}=\xi_r\iota_{n,r}(\gamma,\delta)\xi_r^{-1}\).
For \(X\in\M_{r,n}(E)\), set
\[
  R_{T,X}
  =
  \begin{pmatrix}
    T  & TX^*   \\
    XT & T[X^*]
  \end{pmatrix},
\]
\[
  P_{T,X}
  =
  P_{\mu;\bfk,\bfl}^{(n,r)}(R_{T,X}),
\]
and
\[
  \calX_{\gamma,\delta}(T)
  =
  \left\{
  X\in\M_{r,n}(E):
  X^*L_\delta\subset L_\gamma,\
  R_{T,X}\in
  \Lambda_{\Xi_{\gamma,\delta}}(\frakN)
  \right\}.
\]
For every finite place \(v\), put \(L_{\gamma,v}=L_\gamma\otimes_{\calO_E}\calO_{E,v}\) and \(L_{\delta,v}=L_\delta\otimes_{\calO_E}\calO_{E,v}\), and let \(\Lambda_{\Xi_{\gamma,\delta},v}(\frakN)\) denote the localization of \(\Lambda_{\Xi_{\gamma,\delta}}(\frakN)\).  Define
\[
  \calX_{\gamma,\delta,v}(T)
  =
  \left\{
  X\in\M_{r,n}(E_v):
  X^*L_{\delta,v}\subset L_{\gamma,v},
  R_{T,X}\in
  \Lambda_{\Xi_{\gamma,\delta},v}(\frakN)
  \right\}.
\]

For each \(v\in\bfa\), choose integers \(a_v,b_v\ge0\) such that \(a_v+b_v=n\), \(a_v\le\nu_v\), and \(b_v\le\kappa_v\),
and set \(a=(a_v)_{v\in\bfa}\) and \(b=(b_v)_{v\in\bfa}\).
Choose a splitting character \(\chi_T:E^\times\backslash\bbA_E^\times \rightarrow\bbC^\times\) for the Hermitian space \((E^n,T)\) such that \(\chi_T|_{\bbA_F^\times}=\chi_{E/F}^{n}\),
and \(\chi_{T,v}(z)=\left(z/|z|\right)^{a_v-b_v}\) for every \(v\in\bfa\);
see \cite{Kudla1994Splitting}.
Let \(\omega_T=\bigotimes_v\omega_{T,v}\) be the corresponding Weil representation,
and put \(\chi^{[n]}=\chi\chi_T\), \(\kappa_v^{[n]}=\kappa_v-b_v\), and \(\nu_v^{[n]}=\nu_v-a_v\).

Fix \(\delta\in\calG^{(r)}(\frakN)\).
For \(v\in\bfh\),
put
\[
  K_{\delta,v}
  =
  G_r(F_v)\cap
  \delta_vK_{r,v}(\frakN_v)\delta_v^{-1},
\]
and
\[
  \phi_{\delta,v}
  =
  \mathbbm1_{\calX_{\gamma,\delta,v}(T)}.
\]
If \(c=t_{r,0}(a)n_r(B) \in P_{r,0}(F_v)\cap K_{\delta,v}\),
then
\begin{equation}\label{eq:local-weil-transformation}
  \omega_{T,v}(c)\phi_{\delta,v}
  =
  \chi_{T,v}(\det a)\phi_{\delta,v}.
\end{equation}
Indeed, the Levi action preserves the two conditions defining \(\calX_{\gamma,\delta,v}(T)\),
and hence preserves its characteristic function.
For the unipotent part, the Schr\"odinger model gives
\[
  \bigl(\omega_{T,v}(n_r(B))\phi_{\delta,v}\bigr)(X)
  =
  \psi_{F_v}\!\left(\tr(T[X^*]B)\right)
  \phi_{\delta,v}(X).
\]
If \(X\in\calX_{\gamma,\delta,v}(T)\), then
\(R_{T,X}\in\Lambda_{\Xi_{\gamma,\delta},v}(\frakN)\).
Since \(n_r(B)\) belongs to the unipotent subgroup determined by
\(K_{\delta,v}\), the defining duality for
\(\Lambda_{\Xi_{\gamma,\delta},v}(\frakN)\) gives
\[
  \psi_{F_v}\!\left(\tr(T[X^*]B)\right)=1.
\]
The standard formula for the Levi action in the Weil representation
therefore gives
\[
  \omega_{T,v}(c)\phi_{\delta,v}
  =
  \chi_{T,v}(\det a)\phi_{\delta,v},
\]

Put
\[
  W_{\delta,v}
  =
  \Span_{\bbC}
  \left\{
  \omega_{T,v}(k)\phi_{\delta,v}:
  k\in K_{\delta,v}
  \right\},
\]
and \(d_v'=\dim_\bbC W_{\delta,v}\).
Choose a basis \(\phi_{\delta,v,1},\ldots,\phi_{\delta,v,d_v'}\) and its dual basis \(\lambda_{\delta,v,1},\ldots, \lambda_{\delta,v,d_v'}\).

For \(u\in\bbC\),
define \(\varepsilon_{\delta,v,i}(\,\cdot\,,u)\in I_{r,v}(u,\chi_v^{[n]})\) by
\[
  \varepsilon_{\delta,v,i}
  \bigl(
  t_{r,0}(a)n_r(B)k,u
  \bigr)
  =
  |\det(a^*a)|_v^u
  \chi_v^{[n]}(\det a^*)
  \lambda_{\delta,v,i}
  \bigl(
  \omega_{T,v}(k^{-1})
  \phi_{\delta,v}
  \bigr)
\]
for \(a\in\GL_r(E_v)\), \(B\in\Her_r(E_v)\) and \(k\in K_{\delta,v}\),
and set it equal to \(0\) outside \(P_{r,0}(F_v)K_{\delta,v}\).
Equation~\eqref{eq:local-weil-transformation} shows that this is
well defined.  At \(u=m_0/2\),
\begin{equation}\label{eq:dual-basis-identity}
  \sum_{i=1}^{d_v'}
  \varepsilon_{\delta,v,i}(k,m_0/2)
  \omega_{T,v}(k)\phi_{\delta,v,i}
  =
  \phi_{\delta,v}
\end{equation}
for \(k\in K_{\delta,v}\).

We write \(\mathbf i=(i_v)_{v\in\bfh}\) for a tuple satisfying \(1\le i_v\le d_v'\) for every \(v\in\bfh\).
Since \(d_v'=1\) for almost all \(v\), only finitely many choices of \(\mathbf i\) occur.
For \(\mathbf i=(i_v)_{v\in\bfh}\), set
\[
  \varepsilon_{\delta,\mathbf i}^{\bfh}(g_\bfh,u)
  =
  \prod_{v\in\bfh}
  \varepsilon_{\delta,v,i_v}(g_v,u).
\]
At the archimedean places, define
\[
  \varepsilon_\infty^{[n]}(g_\infty,u)
  =
  \prod_{v\in\bfa}
  \det(g_v)^{\nu_v^{[n]}}
  |\delta_r(g_v)|^{
  \kappa_v^{[n]}+\nu_v^{[n]}-2u
  }
  \delta_r(g_v)^{
  -\kappa_v^{[n]}-\nu_v^{[n]}
  },
\]
and
\[
  \varepsilon_{\delta,\mathbf i}^{[n]}(g,u)
  =
  \varepsilon_{\delta,\mathbf i}^{\bfh}(g_\bfh,u)
  \varepsilon_\infty^{[n]}(g_\infty,u).
\]
The associated Siegel Eisenstein series is
\[
  E_{\delta,\mathbf i}^{[n]}(g,u)
  =
  \sum_{\eta\in
    P_{r,0}(F)\backslash G_r(F)}
  \varepsilon_{\delta,\mathbf i}^{[n]}(\eta g,u),
\]
and we write
\[
  E_{\delta,\mathbf i}^{[n]}(g)
  =
  E_{\delta,\mathbf i}^{[n]}(g,m_0/2).
\]
When \(u=m_0/2\) lies outside the region of absolute convergence,
the right-hand side is defined by holomorphic continuation.
Since \(m_0=\mu-n\ge r\),
\cite[Proposition~17.7]{Shimura2000Arithmeticity} gives regularity at \(u=m_0/2\),
except possibly when \((\mu,F,\chi_0)=(n+r+1,\bbQ,\chi_{E/F}^{n+r+1})\) which is excluded in Theorem~\ref{thm:global-pullback}.

For \(X_\infty=(X_v)_{v\in\bfa}\in\prod_{v\in\bfa}\M_{r,n}(\bbC)\),
put
\[
  \varphi_{T,\infty}(X_\infty)
  =
  P_{\mu;\bfk,\bfl}^{(n,r)}
  (R_{T,X_\infty})
  \exp\!
  \left(
  -2\pi
  \tr_\infty(T[X_\infty^*])
  \right).
\]
This is a
\(V_{\sigma_n}\otimes V_{\sigma_r'}\)-valued Schwartz function.
For \(h=(h_\bfh,h_\infty)\in G_r(\bbA_F)\) and \(\phi_{\delta,\mathbf i}=\bigotimes_{v\in\bfh}\phi_{\delta,v,i_v}\), define
\begin{equation}\label{eq:general-r-theta-definition}
  \Theta_{\gamma,T,\delta}^{(r)}
  (h;\phi_{\delta,\mathbf i})
  =
  \sum_{X\in\M_{r,n}(E)}
  \bigl(
  \omega_{T,\bfh}(h_\bfh)
  \phi_{\delta,\mathbf i}
  \bigr)(X)
  \bigl(
  \omega_{T,\infty}(h_\infty)
  \varphi_{T,\infty}
  \bigr)(X),
\end{equation}
where the Weil representation acts trivially on the coefficient space
\(V_{\sigma_n}\otimes V_{\sigma_r'}\).

The pluriharmonicity and covariance of
\(P_{\mu;\bfk,\bfl}^{(n,r)}\) in
Proposition~\ref{prop:differential-criterion} imply
\[
  \frakp_{r,\infty}^{-}
  \Theta_{\gamma,T,\delta}^{(r)}
  (\,\cdot\,;\phi_{\delta,\mathbf i})
  =0
\]
and
\[
  \Theta_{\gamma,T,\delta}^{(r)}
  (hk_\infty;\phi_{\delta,\mathbf i})
  =
  \bigl(
  \tau_{b,a}\otimes\sigma_r'
  \bigr)(k_\infty)^{-1}
  \Theta_{\gamma,T,\delta}^{(r)}
  (h;\phi_{\delta,\mathbf i})
\]
for \(k_\infty\in K_{r,\infty}\).

Finally, for \(u\in\bbC\), define
\[
  \Phi_{\gamma,T,\delta}(\,\cdot\,,u)
  =
  \sum_{\mathbf i}
  \Theta_{\gamma,T,\delta}^{(r)}
  (\,\cdot\,;\phi_{\delta,\mathbf i})
  E_{\delta,\mathbf i}^{[n]}(\,\cdot\,,u),
\]
and put
\[
  \Phi_{\gamma,T}(\,\cdot\,,u)
  =
  \bigl(
  \Phi_{\gamma,T,\delta}(\,\cdot\,,u)
  \bigr)_{\delta\in\calG^{(r)}(\frakN)},
  \qquad
  \Phi_{\gamma,T}
  =
  \bigl(
  \Phi_{\gamma,T,\delta}(\,\cdot\,,m_0/2)
  \bigr)_{\delta\in\calG^{(r)}(\frakN)}.
\]
Since \((\tau_{b,a}\otimes\sigma_r')\otimes\tau_{\kappa^{[n]},\nu^{[n]}}=\tau_{\kappa,\nu}\otimes\sigma_r'=\rho_r'\) and
\eqref{eq:dual-basis-identity} gives the required invariance at the finite places, we have
\[
  \Phi_{\gamma,T}
  \in
  V_{\sigma_n}\otimes_{\bbC}
  \calA_r(\rho_r',\frakN).
\]

For later use, we record the Fourier expansion of \(\Phi_{\gamma,T}\).
For \(Q_1\in\Her_r(E)\) with \(Q_1\ge0\), put
\[
  B_{\gamma,T,\delta,\mathbf i}(Q_1)
  =
  \sum_{\substack{X\in\M_{r,n}(E)\\T[X^*]=Q_1}}
  P_{T,X}\phi_{\delta,\mathbf i}(X).
\]
Then, for \(Q>0\),
\begin{equation}\label{eq:general-r-Fourier-convolution}
  A_{\Phi_{\gamma,T}}(\delta,Q)
  =
  \sum_{\mathbf i}
  \sum_{\substack{Q_1+Q_2=Q\\Q_1,Q_2\ge0}}
  B_{\gamma,T,\delta,\mathbf i}(Q_1)
  A_{E_{\delta,\mathbf i}^{[n]}}(Q_2).
\end{equation}
The sums are finite.

\subsection{Primitive Fourier coefficient formula}\label{subsec:primitive-formula}

For a finite-adelic cusp representative \(g\in G_n(\bbA_{F,\bfh})\) and a positive definite Fourier index \(S\),
write \(A_{\Kl}(g,S)\) and \(A_{\Si}(g,S)\) for the Fourier coefficients of \([f]_{\kappa,\nu,\frakN}^{n}(\,\cdot\,;\chi)\) and \(\widetilde E_{n,\kappa,\nu,\frakN}(\,\cdot\,;\chi)\),
respectively.
Their primitive coefficients in the sense of Definition~\ref{def:primitive-coefficient} are denoted by \(A_{\Kl}^{\prim}(\gamma,S)\) and \(A_{\Si}^{\prim}(\gamma,S)\).
Recall also that \(c_f=C_{n+r,\kappa,\nu}^{(\frakN)}(f;\chi)/(f^\dagger,f^\dagger)_{G_r}\).

\begin{thm}\label{thm:main-petersson-formula}
  Assume \(\mu\ge n+r\).
  If \(\mu=n+r+1\) and \(F=\bbQ\), assume \(\chi_0\ne\chi_{E/F}^{n+r+1}\).
  If \(\mu=n+r\) and \(n=r+1\), assume that \(L^{(\frakN)}(1,f,\St)\) is holomorphic and nonzero.
  Then, for every \(\gamma\in\calC^{(n)}(\frakN)\) and
  \(T\in\Lambda_\gamma(\frakN)_{>0}\),
  \begin{equation}\label{eq:general-r-arithmetic-primitive}
    c_fA_{\Kl}^{\prim}(\gamma,T)
    =
    A_{\Si}^{\prim}(\gamma,T)
    \frac{(\Phi_{\gamma,T},f^\dagger)_{G_r}}
    {(f^\dagger,f^\dagger)_{G_r}}.
  \end{equation}
\end{thm}

\begin{proof}
  Put \(L=L_{\gamma,\frakN}^{(n)}\) and
  \(\calP=\calP_\gamma(T)\).
  For \(\Re(s)>n+r\), set
  \[
    \widetilde E_{n+r,\kappa,\nu,\frakN}^{[n],\theta}(g,s;\chi)
    =
    \bbL^{(\frakN)}(n,\mu;\chi)
    E_{n+r,\kappa,\nu,\frakN}
    (g\theta_{n,r}^{-1},s;\chi)
  \]
  and
  \[
    \calE_{\frakN}(g,h;s)
    =
    (2\pi\sqrt{-1})^{-d_{\bfk,\bfl}}
    \bbD_{\bfk,\bfl}
    \widetilde E_{n+r,\kappa,\nu,\frakN}^{[n],\theta}
    (\iota_{n,r}(g,h),s;\chi).
  \]
  Thus
  \(\calE_{\frakN}(g,h;\mu/2)=\calE_{\frakN}(g,h)\).

  By Theorem~\ref{thm:global-pullback} and
  Definition~\ref{def:primitive-coefficient},
  \begin{align}
     & C_{n+r,\kappa,\nu}^{(\frakN)}(f;\chi)
    A_{\Kl}^{\prim}(\gamma,T)
    \notag                                   \\
     & =
    e^{2\pi\tr_\infty(T)}
    \sum_{M\in\calP}
    \mu_\calP(L,M)
    \int_{[\Her_n]}
    \bigl(
    \calE_{\frakN}
    (n_n(U)t_{n,0}(h_M)\gamma,*;\mu/2),
    f^\dagger
    \bigr)_{G_r}
    \overline{\psi_F(\tr(TU))}
    \,dU.
    \label{eq:primitive-pullback-sum}
  \end{align}

  We use the global pullback identity after continuation to \(s=\mu/2\).
  At this point the differentiated pullback kernel is holomorphic and of moderate growth.
  Since \(f^\dagger\) is cuspidal, the Petersson integral converges absolutely.
  The \(U\)-integration is over a compact quotient,
  and the Fourier expansions below may be inserted first on truncated Siegel sets and then integrated term by term;
  the limit follows from the rapid decay of \(f^\dagger\).

  For \(\delta\in\calG^{(r)}(\frakN)\), put
  \[
    \fraku_\delta
    =
    \fraku_\delta(\frakN)
    =
    \{V\in\Her_r(E):
    n_r(V)\in G_r(F)\cap
    \delta K_r(\frakN)\delta^{-1}\}.
  \]
  Let
  \[
    \calA_\delta
    =
    \{a\in\GL_r(E):
    t_{r,0}(a)\in
    G_r(F)\cap\delta K_r(\frakN)\delta^{-1}\},
  \]
  let \(Z_\delta\) be the kernel of the action
  \(Y\mapsto aYa^*\) of \(\calA_\delta\) on
  \((\Her_r(\bbC)_{>0})^\bfa\), put
  \(e_\delta=\#Z_\delta\), and choose a fundamental domain
  \(\calF_\delta\) for \(\calA_\delta/Z_\delta\).

  Fix \(M\in\calP\), and put
  \[
    \Xi_{M,\delta}
    =
    \xi_r\iota_{n,r}
    \bigl(t_{n,0}(h_M)\gamma,\delta\bigr)\xi_r^{-1}.
  \]

  By \cite[Fact~19 and Proposition~20]{Takeda2026PullbackFormula},
  after unfolding and pairing with \(f^\dagger\),
  only the open orbit represented by \(\xi_r\) contributes.
  We first describe the additive quotient occurring in this orbit.
  Put
  \(J=\widetilde I_r
  =
  \begin{pmatrix}
    0 \\
    I_r
  \end{pmatrix}
  \in\M_{n,r}(E)\).
  For \(U_0\in\Her_n\) and \(V_0\in\Her_r\), set
  \[
    S=
    \begin{pmatrix}
      0   & J \\
      J^* & 0
    \end{pmatrix},
    \qquad
    W=
    \begin{pmatrix}
      U_0 & 0   \\
      0   & V_0
    \end{pmatrix}.
  \]
  Direct multiplication gives
  \[
    \xi_r
    \iota_{n,r}\bigl(n_n(U_0),n_r(V_0)\bigr)
    \xi_r^{-1}
    =
    \begin{pmatrix}
      I-WS & W    \\
      -SWS & I+SW
    \end{pmatrix},
  \]
  where \(SWS
  =
  \begin{pmatrix}
    JV_0J^* & 0       \\
    0       & J^*U_0J
  \end{pmatrix}\).
  Hence
  \[
    \xi_r
    \iota_{n,r}\bigl(n_n(U_0),n_r(V_0)\bigr)
    \xi_r^{-1}
    \in P_{n+r,0}
  \]
  if and only if \(V_0=0\), \(J^*U_0J=0\).
  Writing
  \[
    U_0=
    \begin{pmatrix}
      U_{11}   & U_{12} \\
      U_{12}^* & U_{22}
    \end{pmatrix},
  \]
  the condition \(J^*U_0J=0\) is equivalent to \(U_{22}=0\).
  Therefore the map \(n_n(U_0)\longmapsto U_{22}\)
  identifies \((N_n\cap P_{n,r})\backslash N_n\simeq\Her_r\).
  We write \(C\) for this quotient coordinate.
  Accordingly, the corresponding adelic Fourier integration is over \(C\in{[\Her_r]}\).

  Write \(Z=V+\sqrt{-1}Y\in\prod_{v\in\bfa}\calH_r\) with \(V\in(\Her_r(\mathbf R))^\bfa\) and \(Y\in(\Her_r(\mathbf R)_{>0})^\bfa\),
  and
  \[
    R=
    \begin{pmatrix}
      A   & B \\
      B^* & Q
    \end{pmatrix}
    \in\Her_{n+r}(E).
  \]
  On the component represented by \(\delta\), we have
  \[
    f_\delta^\dagger(V+\sqrt{-1}Y)
    =
    \sum_{Q'>0}
    A_{f^\dagger}(\delta,Q')\,
    e^{2\pi\sqrt{-1}\tr_\infty(Q'V)}
    e^{-2\pi\tr_\infty(Q'Y)}.
  \]
  The \(R\)-th Fourier term of the differentiated
  degree-\((n+r)\) Siegel Eisenstein series is
  \[
    A_{\widetilde E_{n+r,\kappa,\nu,\frakN}^{[n],\theta}}
    (\Xi_{M,\delta},R)\,
    P_{\mu;\bfk,\bfl}^{(n,r)}(R)
    e^{2\pi\sqrt{-1}\tr_\infty(AU+QV)}
    e^{-2\pi\tr_\infty(A+QY)},
  \]
  together with its character on the remaining open-orbit additive
  quotient.

  For \(R=
  \begin{pmatrix}
    A   & B \\
    B^* & Q
  \end{pmatrix}\),
  put \(X=B^*T^{-1}\).
  For \(C\in\Her_r(\bbA_E)\),
  define
  \[
    \calC_{T,B}(C)
    =
    \begin{pmatrix}
      X^*CX & -X^*C \\
      -CX   & C
    \end{pmatrix}.
  \]
  Then
  \begin{align}
    \tr\bigl(R\calC_{T,B}(C)\bigr)
     & =
    \tr\!\left(
    \bigl(
    Q-XTX^*+X(A-T)X^*
    \bigr)C
    \right)                                      \notag \\
     & =
    \tr\!\left(
    \bigl(
    Q-B^*T^{-1}B
    +B^*T^{-1}(A-T)T^{-1}B
    \bigr)C
    \right).
    \label{eq:open-orbit-C-character}
  \end{align}
  Substituting these Fourier expansions into
  \eqref{eq:primitive-pullback-sum}, and decomposing the Petersson
  integral into its \(\delta\)-components, gives
  \begin{align}
     & C_{n+r,\kappa,\nu}^{(\frakN)}(f;\chi)
    A_{\Kl}^{\prim}(\gamma,T)
    \notag                                   \\
     & =
    \sum_{M\in\calP}
    \mu_\calP(L,M)
    \sum_{\delta\in\calG^{(r)}(\frakN)}
    \frac{\vol(K_r(\frakN))}{e_\delta}
    \sum_R\sum_{Q'>0}
    e^{2\pi\tr_\infty(T)}
    e^{-2\pi\tr_\infty(A)}
    \notag                                   \\
     & \quad\times
    A_{\widetilde E_{n+r,\kappa,\nu,\frakN}^{[n],\theta}}
    (\Xi_{M,\delta},R)
    P_{\mu;\bfk,\bfl}^{(n,r)}(R)
    \notag                                   \\
     & \quad\times
    \int_{{[\Her_n]}}
    \int_{{[\Her_r]}}
    \psi_F\!\left(
    \tr\bigl((A-T)(U+X^*CX)\bigr)
    +
    \tr\bigl((Q-XTX^*)C\bigr)
    \right)
    \,dC\,dU
    \notag                                   \\
     & \quad\times
    \int_{\calF_\delta}
    \int_{\fraku_\delta\backslash(\Her_r(\bbC))^\bfa}
    e^{2\pi\sqrt{-1}\tr_\infty((Q-Q')V)}
    e^{-2\pi\tr_\infty((Q+Q')Y)}
    \notag                                   \\
     & \qquad\qquad\times
    \left(
    1\otimes
    \left\langle
    \rho_r'
    \bigl((Y_v,{}^tY_v)_{v\in\bfa}\bigr)\,\cdot,
    A_{f^\dagger}(\delta,Q')
    \right\rangle_{\sigma_r'}
    \right)
    \,dV\,d\mu_r(Y),
    \label{eq:primitive-full-Fourier-substitution}
  \end{align}
  where \(d\mu_r(Y) = \prod_{v\in\bfa} \det(Y_v)^{-2r}\,dY_v\).

  For fixed \(C\), the translation \(U\mapsto U+X^*CX\) preserves the Haar measure on \({[\Her_n]}\).
  Hence Fourier orthogonality in \(U\) gives
  \[
    A=T.
  \]
  The remaining \(C\)-character is therefore \(C\mapsto \psi_F\!\left(\tr\bigl((Q-XTX^*)C\bigr)\right)\),
  and Fourier orthogonality in \(C\) gives
  \[
    Q=T[X^*].
  \]

  Finally, Fourier orthogonality in \(V\) forces
  \[
    Q'=Q=T[X^*].
  \]
  Since \(Q'>0\) and \(T>0\), this is equivalent to \(\rank_E(X)=r\).

  Hence, if we put
  \[
    c_\delta
    =
    \frac{\vol(K_r(\frakN))}{e_\delta}
    \covol_{dV}(\fraku_\delta),
  \]
  then \eqref{eq:primitive-full-Fourier-substitution} becomes
  \begin{align}
     & C_{n+r,\kappa,\nu}^{(\frakN)}(f;\chi)
    A_{\Kl}^{\prim}(\gamma,T)
    \notag                                   \\
     & =
    \sum_{M\in\calP}\mu_\calP(L,M)
    \sum_{\delta\in\calG^{(r)}(\frakN)}
    c_\delta
    \sum_{\substack{
    X\in\M_{r,n}(E)                          \\
        \rank_E(X)=r
      }}
    A_{\widetilde E_{n+r,\kappa,\nu,\frakN}^{[n],\theta}}
    (\Xi_{M,\delta},R_{T,X})
    \notag                                   \\
     & \quad\times
    \int_{\calF_\delta}
    e^{-4\pi\tr_\infty(T[X^*]Y)}
    \left(
    1\otimes
    \left\langle
    \rho_r'
    \bigl((Y_v,{}^tY_v)_{v\in\bfa}\bigr)\,\cdot,
    A_{f^\dagger}\bigl(\delta,T[X^*]\bigr)
    \right\rangle_{\sigma_r'}
    \right)
    P_{T,X}\,
    d\mu_r(Y).
    \label{eq:primitive-after-Fourier-orthogonality}
  \end{align}

  Moreover,
  \[
    (2\pi\sqrt{-1})^{-d_{\bfk,\bfl}}
    \bbD_{\bfk,\bfl}
    e^{2\pi\sqrt{-1}\tr_\infty(R_{T,X}Z)}
    =
    P_{T,X}
    e^{2\pi\sqrt{-1}\tr_\infty(R_{T,X}Z)}.
  \]

  Fix \(M\in\calP\), \(\delta\in\calG^{(r)}(\frakN)\),
  and a Fourier index \(R_{T,X}\) occurring in \eqref{eq:primitive-after-Fourier-orthogonality}.
  Put \(L_\delta=L_{\delta,\frakN}^{(r)}\).
  Since \(R_{T,X}
    =
    {}^t\!\begin{pmatrix}
    I_n & X
  \end{pmatrix}
  T
  \begin{pmatrix}
    I_n & X^*
  \end{pmatrix}\),
  we have
  \[
    \begin{pmatrix}x^*&y^*\end{pmatrix}
    R_{T,X}
    \begin{pmatrix}x'\\y'\end{pmatrix}
    =
    (x+X^*y)^*T(x'+X^*y')
  \]
  for \(x,x'\in M\) and \(y,y'\in L_\delta\).
  Hence the map
  \begin{eqnarray*}
    M\oplus L_\delta&\longrightarrow& N,\\(x,y)&\longmapsto& x+X^*y
  \end{eqnarray*}
  induces
  \begin{equation}\label{eq:primitive-quotient-lattice}
    \frac{M\oplus L_\delta}
    {(M\oplus L_\delta)\cap\ker R_{T,X}}
    \cong
    M+X^*L_\delta,
  \end{equation}
  under which the nondegenerate Hermitian form induced by
  \(R_{T,X}\) is identified with \(T\).

  We therefore put
  \[
    N=M+X^*L_\delta.
  \]
  If the corresponding finite-adelic Fourier coefficient is nonzero,
  then \(R_{T,X}\) is integral on \(M\oplus L_\delta\),
  and hence \(\Tr_{E/F}(z^*Tz')\in\frakd_F^{-1}\) for \(z,z'\in N\).
  Since \(L\subseteq M\subseteq N\),
  this implies \(N\subseteq L_T^\sharp\).
  For \(v\mid\frakN\), we have \(M_v=L_v\) by the definition of \(\calP_\gamma(T)\).
  Moreover,
  the local Fourier coefficient attached to \(R_{T,X,v}\) vanishes unless \(X^*L_{\delta,v}\subseteq L_v\).
  Hence \(N_v=M_v+X^*L_{\delta,v}=L_v\) for \(v\mid\frakN\).
  Consequently,
  \[
    N\in\calP_\gamma(T).
  \]

  By the singular Fourier coefficient formula \cite[\S\S18--19]{Shimura1997EulerProducts},
  applied to \eqref{eq:primitive-quotient-lattice},
  the nondegenerate quotient is the lattice \(N\) equipped with the Hermitian form \(T\).
  With the normalization of \(\widetilde E_{n+r,\kappa,\nu,\frakN}^{[n],\theta}\) fixed above,
  the formula gives
  \begin{equation}\label{eq:primitive-singular-coefficient}
    A_{\widetilde E_{n+r,\kappa,\nu,\frakN}^{[n],\theta}}
    \left(
    \Xi_{M,\delta},
    R_{T,X}
    \right)
    =
    A_{\Si}
    \bigl(t_{n,0}(h_N)\gamma,T\bigr).
  \end{equation}

  We next sum over \(M\).
  Put \(K=X^*L_\delta\).
  For \(M\in\calP\) with \(M+K\in\calP\),
  Lemma~\ref{lem:primitive-reconstruction},
  applied to the upper interval above \(M+K\),
  gives
  \begin{equation}\label{eq:primitive-upper-reconstruction}
    A_{\Si}
    \bigl(t_{n,0}(h_{M+K})\gamma,T\bigr)
    =
    \sum_{\substack{N'\in\calP\\M+K\subseteq N'}}
    (\tau_{\kappa,\nu})_{T_{N'}}[a_{N'},Z_{N'}]
    A_{\Si}^{\prim}(\gamma[N'],T_{N'}).
  \end{equation}
  Hence
  \begin{align}
     & \sum_{\substack{M\in\calP \\M+K\in\calP}}
    \mu_\calP(L,M)
    A_{\Si}
    \bigl(t_{n,0}(h_{M+K})\gamma,T\bigr)
    \notag                       \\
     & =
    \sum_{N'\in\calP}
    (\tau_{\kappa,\nu})_{T_{N'}}[a_{N'},Z_{N'}]
    A_{\Si}^{\prim}(\gamma[N'],T_{N'})
    \mathbf1_{K\subseteq N'}
    \sum_{L\subseteq M\subseteq N'}
    \mu_\calP(L,M)
    \notag                       \\
     & =
    \mathbf1_{K\subseteq L}
    (\tau_{\kappa,\nu})_{T_L}[a_L,Z_L]
    A_{\Si}^{\prim}(\gamma[L],T_L)
    \notag                       \\
     & =
    \mathbf1_{K\subseteq L}
    A_{\Si}^{\prim}(\gamma,T),
    \label{eq:primitive-mobius-cancellation}
  \end{align}
  where \(\mathbf1_{K_1\subseteq K_2}\) denotes the indicator of the condition \(K_1\subseteq K_2\).
  Here the second equality uses
  \[
    \sum_{L\subseteq M\subseteq N'}
    \mu_\calP(L,M)
    =
    \delta_{L,N'}.
  \]
  Thus the M\"obius sum leaves the condition \(X^*L_\delta\subseteq L\).
  Together with the condition \(R_{T,X}\in\Lambda_{\Xi_{\gamma,\delta}}(\frakN)\) already imposed on the Fourier index,
  this gives \(X\in\calX_{\gamma,\delta}(T)\).
  Therefore \eqref{eq:primitive-after-Fourier-orthogonality} becomes
  \begin{align}
     & C_{n+r,\kappa,\nu}^{(\frakN)}(f;\chi)
    A_{\Kl}^{\prim}(\gamma,T)
    \notag                                   \\
     & =
    A_{\Si}^{\prim}(\gamma,T)
    \sum_{\delta\in\calG^{(r)}(\frakN)}
    c_\delta
    \sum_{\substack{
    X\in\calX_{\gamma,\delta}(T)             \\
        \rank_E(X)=r
      }}
    \int_{\calF_\delta}
    e^{-4\pi\tr_\infty(T[X^*]Y)}
    \notag                                   \\
     & \qquad\qquad\times
    \left(
    1\otimes
    \left\langle
    \rho_r'
    \bigl((Y_v,{}^tY_v)_{v\in\bfa}\bigr)\,\cdot,
    A_{f^\dagger}\bigl(\delta,T[X^*]\bigr)
    \right\rangle_{\sigma_r'}
    \right)
    P_{T,X}\,
    d\mu_r(Y).
    \label{eq:primitive-after-mobius}
  \end{align}

  We now show that
  \begin{align}
    (\Phi_{\gamma,T},f^\dagger)_{G_r}= & \sum_{\delta\in\calG^{(r)}(\frakN)}
    c_\delta
    \sum_{\substack{
    X\in\calX_{\gamma,\delta}(T)                                             \\
        \rank_E(X)=r
      }}
    \int_{\calF_\delta}
    e^{-4\pi\tr_\infty(T[X^*]Y)}
    \notag                                                                   \\
                                       & \qquad\times
    \left(
    1\otimes
    \left\langle
    \rho_r'
    \bigl((Y_v,{}^tY_v)_{v\in\bfa}\bigr)\,\cdot,
    A_{f^\dagger}\bigl(\delta,T[X^*]\bigr)
    \right\rangle_{\sigma_r'}
    \right)
    P_{T,X}\,
    d\mu_r(Y).
    \label{eq:primitive-open-sum-after-mobius}
  \end{align}

  Put \(u_0=(\mu-n)/2\).
  For \(\Re(u)>r\), the Eisenstein series \(E_{\delta,\mathbf i}^{[n]}(\,\cdot\,,u)\) is absolutely convergent.
  Decomposing the Petersson pairing into its \(\delta\)-components and unfolding \(E_{\delta,\mathbf i}^{[n]}(\,\cdot\,,u)\),
  we obtain
  \begin{align}
     & \bigl(
    \Phi_{\gamma,T}(\,\cdot\,,u),
    f^\dagger
    \bigr)_{G_r}
    \notag    \\
     & =
    \sum_{\delta\in\calG^{(r)}(\frakN)}
    \frac{\vol(K_r(\frakN))}{e_\delta}
    \sum_{\mathbf i}
    \int_{P_{r,0}(F)\backslash G_r(\bbA_F)}
    \left\langle
    \Theta_{\gamma,T,\delta}^{(r)}
    (h;\phi_{\delta,\mathbf i})
    \varepsilon_{\delta,\mathbf i}^{[n]}(h,u),
    f^\dagger(h)
    \right\rangle_r
    \,dh.
    \label{eq:primitive-Phi-unfolding}
  \end{align}

  The identity obtained for \(\Re(u)>r\) continues to \(u=u_0\).
  Indeed, after pairing with the cusp form \(f^\dagger\),
  rapid decay gives absolute convergence at \(u_0\),
  and the unfolded integral represents the same continuation.
  At this point, using the Iwasawa decomposition on the
  \(\delta\)-component and the definition
  \eqref{eq:general-r-theta-definition}, the right-hand side becomes
  \begin{align}
     & \bigl(
    \Phi_{\gamma,T},
    f^\dagger
    \bigr)_{G_r}
    \notag          \\
     & =
    \sum_{\delta\in\calG^{(r)}(\frakN)}
    \frac{\vol(K_r(\frakN))}{e_\delta}
    \int_{\calF_\delta}
    \int_{\fraku_\delta\backslash(\Her_r(\bbC))^\bfa}
    \sum_{X\in\M_{r,n}(E)}
    \Biggl(
    \sum_{\mathbf i}
    \varepsilon_{\delta,\mathbf i}^{\bfh}
      (k_\bfh,u_0)
    \bigl(
      \omega_{T,\bfh}(k_\bfh)
      \phi_{\delta,\mathbf i}
      \bigr)(X)
    \Biggr)
    \notag          \\
     & \qquad\times
    e^{2\pi\sqrt{-1}\tr_\infty(T[X^*]V)}
    e^{-2\pi\tr_\infty(T[X^*]Y)}
    \notag          \\
     & \qquad\times
    \left\langle
    \rho_r'
    \bigl((Y_v,{}^tY_v)_{v\in\bfa}\bigr)
    P_{T,X},
    f_\delta^\dagger(V+\sqrt{-1}Y)
    \right\rangle_{\sigma_r'}
    \,dV\,d\mu_r(Y).
    \label{eq:primitive-Phi-Iwasawa}
  \end{align}
  By \eqref{eq:dual-basis-identity},
  \[
    \sum_{\mathbf i}
    \varepsilon_{\delta,\mathbf i}^{\bfh}
    (k_\bfh,u_0)
    \bigl(
    \omega_{T,\bfh}(k_\bfh)
    \phi_{\delta,\mathbf i}
    \bigr)(X)
    =
    \mathbbm1_{\calX_{\gamma,\delta}(T)}(X).
  \]
  Hence
  \begin{align}
     & \bigl(
    \Phi_{\gamma,T},
    f^\dagger
    \bigr)_{G_r}
    \notag          \\
     & =
    \sum_{\delta\in\calG^{(r)}(\frakN)}
    \frac{\vol(K_r(\frakN))}{e_\delta}
    \sum_{X\in\calX_{\gamma,\delta}(T)}
    \int_{\calF_\delta}
    \int_{\fraku_\delta\backslash(\Her_r(\bbC))^\bfa}
    e^{2\pi\sqrt{-1}\tr_\infty(T[X^*]V)}
    e^{-2\pi\tr_\infty(T[X^*]Y)}
    \notag          \\
     & \qquad\times
    \left\langle
    \rho_r'
    \bigl((Y_v,{}^tY_v)_{v\in\bfa}\bigr)
    P_{T,X},
    f_\delta^\dagger(V+\sqrt{-1}Y)
    \right\rangle_{\sigma_r'}
    \,dV\,d\mu_r(Y) \\
     & =
    \sum_{\delta\in\calG^{(r)}(\frakN)}
    \frac{\vol(K_r(\frakN))}{e_\delta}
    \sum_{X\in\calX_{\gamma,\delta}(T)}
    \sum_{Q'>0}
    \int_{\calF_\delta}
    \int_{\fraku_\delta\backslash(\Her_r(\bbC))^\bfa}
    e^{2\pi\sqrt{-1}
        \tr_\infty((T[X^*]-Q')V)}
    \notag          \\
     & \qquad\times
    e^{-2\pi\tr_\infty((T[X^*]+Q')Y)}
    \left(
    1\otimes
    \left\langle
    \rho_r'
    \bigl((Y_v,{}^tY_v)_{v\in\bfa}\bigr)\,\cdot,
    A_{f^\dagger}(\delta,Q')
    \right\rangle_{\sigma_r'}
    \right)
    P_{T,X}
    \,dV\,d\mu_r(Y).
    \label{eq:primitive-Phi-Fourier-substitution}
  \end{align}
  Fourier orthogonality in \(V\) gives
  \[
    Q'=T[X^*].
  \]
  Since \(Q'>0\) and \(T>0\), this is possible exactly when \(\rank_E(X)=r\).
  Moreover,
  \[
    \int_{\fraku_\delta\backslash(\Her_r(\bbC))^\bfa}
    1\,dV
    =
    \covol_{dV}(\fraku_\delta).
  \]
  Therefore,
  \begin{align}
    \bigl(
    \Phi_{\gamma,T},
    f^\dagger
    \bigr)_{G_r}
     & =
    \sum_{\delta\in\calG^{(r)}(\frakN)}
    c_\delta
    \sum_{\substack{
    X\in\calX_{\gamma,\delta}(T) \\
        \rank_E(X)=r
      }}
    \int_{\calF_\delta}
    e^{-4\pi\tr_\infty(T[X^*]Y)}
    \notag                       \\
     & \qquad\times
    \left(
    1\otimes
    \left\langle
    \rho_r'
    \bigl((Y_v,{}^tY_v)_{v\in\bfa}\bigr)\,\cdot,
    A_{f^\dagger}
    \bigl(\delta,T[X^*]\bigr)
    \right\rangle_{\sigma_r'}
    \right)
    P_{T,X}\,
    d\mu_r(Y).
    \label{eq:primitive-Phi-pairing}
  \end{align}

  Comparing \eqref{eq:primitive-Phi-pairing} with
  \eqref{eq:primitive-after-mobius}, we obtain
  \begin{equation}\label{eq:primitive-before-normalization}
    C_{n+r,\kappa,\nu}^{(\frakN)}(f;\chi)
    A_{\Kl}^{\prim}(\gamma,T)
    =
    A_{\Si}^{\prim}(\gamma,T)
    \bigl(
    \Phi_{\gamma,T},
    f^\dagger
    \bigr)_{G_r}.
  \end{equation}

  Finally,
  division by \((f^\dagger,f^\dagger)_{G_r}\) gives
  \[
    c_fA_{\Kl}^{\prim}(\gamma,T)
    =
    A_{\Si}^{\prim}(\gamma,T)
    \frac{
      (\Phi_{\gamma,T},f^\dagger)_{G_r}
    }{
      (f^\dagger,f^\dagger)_{G_r}
    }.
  \]
\end{proof}

\section{Arithmetic applications}\label{sec:arithmetic-applications}

\subsection{Integrality}\label{subsec:arithmetic-integrality}

Fix a number field \(L\subset\bbC\),
stable under complex conjugation,
which contains the coefficient fields of the forms and characters under consideration and the values of \(\psi_{F,\bfh}\) occurring below.
We assume that the orthogonal projection onto the cuspidal subspace with respect to the Petersson inner product is defined over \(L\).
We also assume that the linear functional
\[
  h\longmapsto
  \frac{(h,f^\dagger)_{G_r}}
  {(f^\dagger,f^\dagger)_{G_r}}
\]
is defined over \(L\).
Let \(\frakP\) be a prime of \(L\) above the rational prime \(p\),
and write \(V_{\rho_n,\calO_{L,\frakP}}\) and \(V_{\rho_r',\calO_{L,\frakP}}\) for the \(\calO_{L,\frakP}\)-spans of the standard bideterminant bases in the corresponding representation spaces.
Let \(\calA_n(L)\subset\calA_n(\rho_n,\frakN)\) and \(\calA_r(L)\subset\calA_r(\rho_r',\frakN)\) be the corresponding \(L\)-structures,
and let \(\calA_n(\calO_{L,\frakP})\) and \(\calA_r(\calO_{L,\frakP})\) be the submodules of forms whose Fourier coefficients at the chosen cusps belong to these integral lattices.
Put
\[
  \calS_r(L)=\calA_r(L)\cap\calS_r(\rho_r',\frakN),
  \qquad
  \calS_r(\calO_{L,\frakP})
  =
  \calA_r(\calO_{L,\frakP})\cap\calS_r(\rho_r',\frakN).
\]
Let \(\pr:\calA_r(L)\rightarrow\calS_r(L)\) be the projection characterized by
\[
  (\pr(h),s)_{G_r}=(h,s)_{G_r}
  \qquad
  (h\in\calA_r(L),\ s\in\calS_r(L)).
\]
Define ideals of \(\calO_{L,\frakP}\) by
\[
  \frakd_{\cusp}(r)
  =
  \left\{
  a\in\calO_{L,\frakP}
  \;\middle|\;
  a\,\pr(\calA_r(\calO_{L,\frakP}))
  \subset\calS_r(\calO_{L,\frakP})
  \right\},
\]
\[
  \frakA(f^\dagger)
  =
  \left\{
  a\in\calO_{L,\frakP}
  \;\middle|\;
  a\,
  \frac{(h,f^\dagger)_{G_r}}
  {(f^\dagger,f^\dagger)_{G_r}}
  \in\calO_{L,\frakP}
  \text{ for every }h\in\calS_r(\calO_{L,\frakP})
  \right\}.
\]
By the preceding rationality assumption and finite dimensionality,
both are nonzero ideals of \(\calO_{L,\frakP}\).
Put
\[
  D_{F,E}(\frakN)
  =
  N_{F/\bbQ}(\frakd_F\frakN\frakD_{E/F}),
  \qquad
  d_{\max}
  =
  \max_{v\in\bfa}(|\bfk_v|+|\bfl_v|).
\]

\begin{prop}\label{prop:integral-pullback-normalization}
  Assume \(p\nmid D_{F,E}(\frakN)\),
  \(p>2\mu+d_{\max}-4\),
  and \(p>\mu+1\).
  Then for every cusp representative \(\gamma\) and every Fourier index \(S\),
  \begin{equation}\label{eq:normalized-Klingen-integrality}
    c_fA_{\Kl}(\gamma,S)
    \in
    (\frakd_{\cusp}(r)\frakA(f^\dagger))^{-1}
    V_{\rho_n,\calO_{L,\frakP}}.
  \end{equation}
  Suppose in addition that \((\gamma,T)\) satisfies
  \(M_v=L_{\gamma,\frakN,v}^{(n)}\) for every \(M\in\calP_\gamma(T)\) and every \(v\mid p\),
  and that
  \[
    \frakd_{\cusp}(r)
    =
    \frakA(f^\dagger)
    =
    \calO_{L,\frakP}.
  \]
  Then
  \begin{equation}\label{eq:normalized-primitive-Klingen-integrality}
    c_fA_{\Kl}^{\prim}(\gamma,T)
    \in V_{\rho_n,\calO_{L,\frakP}}.
  \end{equation}
\end{prop}

\begin{proof}
  For
  \(
  \psi_a^\chi=\chi_0^{-1}\chi_{E/F}^{a}
  \)
  with \(0\le a<n\),
  the assumptions give
  \(1\le\mu-a<p-1\).
  Since \(p\nmid D_{F,E}(\frakN)\),
  the characters \(\psi_a^\chi\) are unramified at the places above \(p\).
  The integrality theorem for critical values in \cite[Theorem~0.4]{DeligneRibet1980Values} therefore gives
  \[
    L^{(\frakN)}(1-\mu+a,\psi_a^\chi)
    \in\calO_{L,\frakP}.
  \]
  Indeed, the factor occurring in the Deligne--Ribet integrality theorem may be chosen to be a \(\frakP\)-adic unit since \(0<\mu-a<p-1\).
  Hence
  \[
    \bbL^{(\frakN)}(n,\mu;\chi)
    \in\calO_{L,\frakP}.
  \]
  By the local formulas for Hermitian Siegel series \cite[\S\S13--16]{Shimura1997EulerProducts} (cf. \cite[\S2]{Ikeda2008Lifting}),
  the local Fourier coefficients of \(\widetilde E_{n+r,\kappa,\nu,\frakN}^{[n],\theta}\) at every finite place \(v\) are integral at \(\frakP\).
  Moreover,
  \cite[(4.8), Theorem~5.11, Example~5.12]{Takeda2025DifferentialOperators} shows,
  under the assumed bound on \(p\),
  that \((2\pi\sqrt{-1})^{-d_{\bfk,\bfl}}\bbD_{\bfk,\bfl}\) maps integral Fourier expansions to integral Fourier expansions.
  Thus the Fourier coefficients of the differentiated Eisenstein series occurring in the pullback formula are \(\frakP\)-integral.

  Taking the Petersson pairing of \(\calE_{\frakN}(g,\cdot)\) with \(f^\dagger\) in the \(G_r\)-variable and using the definitions of \(\frakd_{\cusp}(r)\) and \(\frakA(f^\dagger)\) gives \eqref{eq:normalized-Klingen-integrality}.

  Assume now the additional hypotheses in the second part of the proposition.
  For \(M\in\calP_\gamma(T)\),
  \eqref{eq:fourier-coefficient-transformation} gives
  \[
    A_{\Kl}(t_{n,0}(h_M)\gamma,T)
    =
    (\rho_n)_{T_M}[a_M,Z_M]
    A_{\Kl}(\gamma[M],T_M).
  \]
  Since \(M_v=L_{\gamma,\frakN,v}^{(n)}\) for every \(v\mid p\),
  the operator \((\rho_n)_{T_M}[a_M,Z_M]\) preserves
  \(V_{\rho_n,\calO_{L,\frakP}}\):
  the matrix \(\rho_n[a_M]\) preserves this lattice and
  \(\psi_{F,\bfh}(\tr(T_MZ_M))\) is a root of unity, hence a \(\frakP\)-adic unit.
  As \(\frakd_{\cusp}(r)=\frakA(f^\dagger)=\calO_{L,\frakP}\),
  equation \eqref{eq:normalized-Klingen-integrality} therefore implies
  \[
    c_f
    A_{\Kl}(t_{n,0}(h_M)\gamma,T)
    \in
    V_{\rho_n,\calO_{L,\frakP}}
  \]
  for every \(M\in\calP_\gamma(T)\).
  Taking the finite M\"obius sum in
  Definition~\ref{def:primitive-coefficient} gives
  \[
    c_fA_{\Kl}^{\prim}(\gamma,T)
    \in
    V_{\rho_n,\calO_{L,\frakP}},
  \]
  which is \eqref{eq:normalized-primitive-Klingen-integrality}.
\end{proof}

\subsection{The degree-two case}\label{subsec:degree-two-arithmetic}

Assume
\(F=\bbQ\),
\(n=2\),
\(r=1\),
\(\frakN=(1)\),
and \(\chi=1\),
and write \(D_E=|\operatorname{disc}(E)|\).
Let \(m\ge3\) and \(w>m\),
and put \(\kappa=\nu=m\), \(\bfk=\bfl=(w-m)\), \(\mu=2m\).
Let \(f_0(q)=\sum_{n\ge1}a(n;f_0)q^n\in S_{2w}(\SL_2(\bbZ))\) be a Hecke eigenform normalized by \(a(1;f_0)=1\),
and let \(f\) be the corresponding full-level degree-one Hermitian cusp form.

\begin{prop}\label{prop:degree-two-integral-pullback}
  Assume that \(p>2w\) and \(p\nmid D_E\).
  For every cusp representative \(\gamma\) and every Fourier index \(S\),
  \[
    c_fA_{\Kl}(\gamma,S)
    \in
    (\frakd_{\cusp}(1)\frakA(f^\dagger))^{-1}
    V_{\rho_2,\calO_{L,\frakP}}.
  \]
\end{prop}
\begin{proof}
  Put \(d=w-m\) and \(\nu_d=\begin{pmatrix}0&d\\d&0\end{pmatrix}\).
  Under the stated assumptions,
  the critical Hecke \(L\)-values and the local factors arising from the
  Siegel series at finite places are \(\frakP\)-integral,
  by the same argument as in the proof of
  Proposition~\ref{prop:integral-pullback-normalization}.
  It remains to consider the differential operator.

  Specializing \cite[Example~4.15]{Takeda2025DifferentialOperators}
  to \(a_1=a_2=d\) and \(\kappa=2m\), we obtain
  \[
    P_{\nu_d,2m}^{M}(S)
    =
    \sum_{i=0}^{d}
    (-1)^i
    \frac{\binom di^2}{\binom{2w-2}{i}}
    s_{11}^{i}s_{12}^{d-i}s_{21}^{d-i}s_{22}^{i},
  \]
  where \(S=(s_{ij})\).

  Applying \cite[Proposition~5.10]{Takeda2025DifferentialOperators}
  to \(P_{\nu_d,2m}^{M}\), we have
  \[
    \phi_{2m}\!\left(P_{\nu_d,2m}^{M}\right)(T)
    =
    c_{m,w}\,t_{12}^{d}t_{21}^{d},
  \]
  where evaluation at
  \(T=\left(\begin{smallmatrix}0&1\\1&0\end{smallmatrix}\right)\)
  gives
  \[
    c_{m,w}
    =
    (2m-1)(2w-1)
    \left(
    \prod_{r=0}^{d-2}(2m+r)
    \right)^2,
  \]
  with the product interpreted as \(1\) when \(d=1\).
  Hence the polynomial \(P_0\) in \cite[Theorem~5.13]{Takeda2025DifferentialOperators} is \(P_0=c_{m,w}^{-1}P_{\nu_d,2m}^{M}\).

  Every prime dividing a denominator of a coefficient of \(P_{\nu_d,2m}^{M}\) is at most \(2w-2\).
  Moreover, every prime dividing \(c_{m,w}\) is smaller than \(2w\).
  Since \(p>2w\),
  both \(P_{\nu_d,2m}^{M}\) and \(c_{m,w}^{-1}\) are \(\frakP\)-integral.
  Thus the normalized differential operator preserves the integral coefficient lattice.
  The assertion follows from the same cuspidal-projection and normalized Petersson-pairing argument as in Proposition~\ref{prop:integral-pullback-normalization}.
\end{proof}

Specializing \eqref{eq:C-explicit-formula} and applying the functional equations gives
\begin{equation}\label{eq:degree-two-cf-standard-L-relation}
  \frac{c_f}
  {L^{\mathrm{alg}}(m-\tfrac12,f,\St)}
  \in
  \calO_{L,\frakP}^{\times}
\end{equation}
when \(p>2w\) and \(p\nmid D_E\),
where
\[
  L^{\mathrm{alg}}\!\left(m-\tfrac12,f,\St\right)
  :=
  \frac{
  L(m-\tfrac12,f,\St)
  }{
  D_E^{1/2}\,
  \pi^{2w+2m-2}\,
  (f,f)_{\U_{1,1}}
  }
\]
is the algebraically normalized critical value of the standard \(L\)-function,
as in \cite[Corollary~5.14]{Takeda2026Congruences}.

For
\(x\in V_{\rho_2,\calO_{L,\frakP}}\otimes_{\calO_{L,\frakP}}L\),
put
\[
  v_{\frakP}(x)
  =
  \max\left\{
  a\in\bbZ
  \;\middle|\;
  x\in(\frakP\calO_{L,\frakP})^a
  V_{\rho_2,\calO_{L,\frakP}}
  \right\},
  \qquad
  v_{\frakP}(0)=+\infty.
\]

\begin{cor}\label{cor:degree-two-primitive-nonvanishing}
  Assume that \(p>2w\), \(p\nmid D_E\), and \(\frakd_{\cusp}(1)=\frakA(f^\dagger)=\calO_{L,\frakP}\).
  Let \((\gamma,T)\) satisfy \(M_v=L_{\gamma,\frakN,v}^{(2)}\) for every \(M\in\calP_\gamma(T)\) and \(v\mid p\).
  Then
  \[
    v_{\frakP}\!\left(
    A_{\Si}^{\prim}(\gamma,T)
    \frac{(\Phi_{\gamma,T},f^\dagger)_{G_1}}
    {(f^\dagger,f^\dagger)_{G_1}}
    \right)
    \ge0.
  \]
  If equality holds,
  then there is an \(M\in\calP_\gamma(T)\) such that
  \begin{equation}\label{eq:degree-two-klingen-nonvanishing}
    v_{\frakP}\!\left(
    L^{\mathrm{alg}}(m-\tfrac12,f,\St)
    A_{\Kl}(\gamma[M],T_M)
    \right)
    =0.
  \end{equation}
\end{cor}

\begin{proof}
  Proposition~\ref{prop:degree-two-integral-pullback} gives
  \(
  c_fA_{\Kl}(\gamma',S)
  \in V_{\rho_2,\calO_{L,\frakP}}
  \)
  for every Fourier index \((\gamma',S)\),
  and the local condition on \((\gamma,T)\) gives the same integrality for
  \(c_fA_{\Kl}^{\prim}(\gamma,T)\).
  By \eqref{eq:degree-two-cf-standard-L-relation},
  the same assertions hold with \(c_f\) replaced by
  \(L^{\mathrm{alg}}(m-\tfrac12,f,\St)\).
  Equation~\eqref{eq:general-r-arithmetic-primitive} therefore gives the first assertion.

  If equality holds,
  \eqref{eq:general-r-arithmetic-primitive} and
  \eqref{eq:degree-two-cf-standard-L-relation} give
  \[
    v_{\frakP}\!\left(
    L^{\mathrm{alg}}(m-\tfrac12,f,\St)
    A_{\Kl}^{\prim}(\gamma,T)
    \right)
    =0.
  \]
  Expanding the primitive coefficient by M\"obius inversion,
  all summands are integral and every
  \((\rho_2)_{T_M}[a_M,Z_M]\) preserves the integral coefficient lattice at \(\frakP\).
  Hence at least one summand has valuation zero,
  which gives \eqref{eq:degree-two-klingen-nonvanishing}.
\end{proof}

\begin{rem}
  Suppose \(h_E=1\).
  Then the full-level degree-one space has a single adelic component.
  By the degree-one Fourier expansion in \cite[Lemma~18.7(2)]{Shimura1997EulerProducts},
  the constant term of the degree-one Eisenstein series used above is a \(\frakP\)-adic unit multiple of \(\zeta(3-2m)\),
  and its positive Fourier coefficients are \(\frakP\)-integral.
  Thus
  \(v_{\frakP}(\zeta(3-2m))=0\)
  implies
  \[
    \frakd_{\cusp}(1)
    =
    \calO_{L,\frakP}.
  \]
  If in addition \(f_0\) is not congruent modulo \(\frakP\) to any other normalized primitive eigenform of weight \(2w\),
  the integral Hecke projector onto the subspace spanned by \(f_0\) gives
  \[
    \frakA(f^\dagger)
    =
    \calO_{L,\frakP}.
  \]
\end{rem}

We next give a criterion involving only finitely many Fourier coefficients for the equality case in
Corollary~\ref{cor:degree-two-primitive-nonvanishing}.
Put \(d=w-m\), \(B=\left\lfloor w/6\right\rfloor\), and \(h=\dim S_{2w}(\SL_2(\bbZ))\).
For \(T=(t_{ab})\in\Her_2(E)_{>0}\) and \(x=(x_1,x_2)\in E^2\),
introduce indeterminates \(u_1,u_2,v_1,v_2\) and put
\[
  A_x=u_1\overline{x}_1+u_2\overline{x}_2,
  \qquad
  B_x=v_1x_1+v_2x_2,
\]
\[
  D_T=\sum_{a,b=1}^2u_at_{ab}v_b,
  \qquad
  C_{T,x}=T[x^*]D_T-A_xB_x.
\]
Define
\begin{equation}\label{eq:degree-two-integral-differential-polynomial}
  \widehat P_{m,w}(T,x)
  =
  \sum_{c=0}^{d}
  (-1)^c
  \binom dc
  \binom{m+w-2}{d-c}
  A_x^{d-c}B_x^{d-c}C_{T,x}^{c}.
\end{equation}
By Theorem~5.13 of \cite{Takeda2025DifferentialOperators},
the polynomial occurring in the differential operator is a nonzero scalar multiple of
\(\widehat P_{m,w}\),
and that scalar is a \(\frakP\)-adic unit when \(p>2w\).

Let \(T_1,\ldots,T_s\) be positive definite integral Hermitian matrices.
For \(1\le i\le s\) and \(1\le q\le B\),
write
\[
  A_{\Si}^{\prim}(1,T_i)
  \sum_{\substack{x\in\calO_E^2\\T_i[x^*]=q}}
  \widehat P_{m,w}(T_i,x)
  =
  \sum_{a,b=0}^{d}
  c_{q,i}(a,b)
  u_1^a u_2^{d-a}v_1^b v_2^{d-b}.
\]
We put
\[
  \calM(T_i)
  =
  \bigl(c_{q,i}(a,b)\bigr)_{
      \substack{1\le q\le B\\0\le a,b\le d}}
  \in
  M_{B,(d+1)^2}(L),
\]
and
\begin{equation}\label{eq:degree-two-general-finite-matrix}
  \calM(T_1,\ldots,T_s)
  =
  \bigl(
  \calM(T_1)
  \mid\cdots\mid
  \calM(T_s)
  \bigr)
  \in
  M_{B,s(d+1)^2}(L).
\end{equation}

\begin{lem}\label{lem:degree-two-rank-criterion}
  Assume that \(p>2w\), \(p\nmid D_E\), \(v_{\frakP}(\zeta(3-2m))=0\)
  and \(\frakd_{\cusp}(1)=\frakA(f^\dagger)=\calO_{L,\frakP}\).
  If
  \begin{equation}\label{eq:degree-two-general-rank-condition}
    \rank_{\calO_{L,\frakP}/\frakP}
    \bigl(
    \calM(T_1,\ldots,T_s)\bmod\frakP
    \bigr)
    =h,
  \end{equation}
  then
  \[
    v_{\frakP}\!\left(
    A_{\Si}^{\prim}(1,T_i)
    \frac{(\Phi_{1,T_i},f^\dagger)_{G_1}}
    {(f^\dagger,f^\dagger)_{G_1}}
    \right)
    =0
  \]
  for some \(1\le i\le s\).
\end{lem}

\begin{proof}
  Write
  \[
    A_{\Si}^{\prim}(1,T_i)\Phi_{1,T_i}
    =
    \sum_{a,b=0}^{d}
    F_{i,a,b}
    u_1^a u_2^{d-a}v_1^b v_2^{d-b}.
  \]
  Since \(d>0\),
  each \(F_{i,a,b}\) is a cusp form in \(S_{2w}(\SL_2(\bbZ))\).
  Write the degree-one Eisenstein series occurring in
  \(\Phi_{1,T_i}\) as
  \[
    E^{(1)}(z)
    =
    \sum_{n\ge0}e(n)q^n.
  \]
  The assumption on \(\zeta(3-2m)\) gives
  \(e(0)\in\calO_{L,\frakP}^{\times}\).
  The first \(B\) Fourier coefficients of the forms \(F_{i,a,b}\)
  are therefore obtained from the columns of
  \(\calM(T_1,\ldots,T_s)\)
  by left multiplication by a lower-triangular matrix whose diagonal entries are \(\frakP\)-adic units.
  Hence \eqref{eq:degree-two-general-rank-condition} implies that the reductions of the \(F_{i,a,b}\) span
  \(S_{2w}(\SL_2(\bbZ))\) modulo \(\frakP\),
  by the Sturm bound.

  Since \(\frakA(f^\dagger)=\calO_{L,\frakP}\), the functional
  \[
    \lambda_f(h)
    =
    \frac{(h,f^\dagger)_{G_1}}
    {(f^\dagger,f^\dagger)_{G_1}}
  \]
  maps \(\calS_1(\calO_{L,\frakP})\) to \(\calO_{L,\frakP}\).
  Since \(f_0\) is normalized by \(a(1;f_0)=1\),
  the corresponding form \(f^\dagger\) belongs to \(\calS_1(\calO_{L,\frakP})\).
  Hence \(\lambda_f(f^\dagger)=1\),
  so the reduction of \(\lambda_f\) modulo \(\frakP\) is nonzero.
  Since the reductions of the \(F_{i,a,b}\) span the cuspidal space,
  \(\lambda_f(F_{i,a,b})\) is a \(\frakP\)-adic unit for at least one triple \((i,a,b)\).
  The corresponding coordinate of
  \[
    A_{\Si}^{\prim}(1,T_i)
    \frac{(\Phi_{1,T_i},f^\dagger)_{G_1}}
    {(f^\dagger,f^\dagger)_{G_1}}
  \]
  is therefore a \(\frakP\)-adic unit.
\end{proof}

\subsection{Examples}\label{subsec:harder-examples}

For the application to Harder's conjecture,
take \(m=j/2+2\) and \(w=k+j/2-1\).
Then \(2w=2k+j-2\) and \(d=w-m=k-3\).
Combined with Corollary~\ref{cor:degree-two-primitive-nonvanishing},
the rank criterion below verifies the Fourier-coefficient condition in \cite[Theorem~5.11]{katsurada2026harder} from finitely many Fourier coefficients.
We specialize to \(E=\bbQ(\sqrt{-3})\),
as in the numerical examples of \cite[Section~6]{katsurada2026harder}.
We choose \(\calC^{(2)}(1)\) so that it contains the identity element of \(G_2(\bbA_{\bbQ,\bfh})\),
and write \(1\) for this representative.
For an elliptic eigenform \(f_0\) of weight \(2w\),
the present full-level Klingen--Eisenstein series agrees with the lift denoted by \([f_0]^2_{2m}\) in \cite[Subsection~3.3]{katsurada2026harder},
where \(f_0\) is identified with its full-level Hermitian adelization on \(G_1\).
We use the algebraic normalization of the elliptic critical values from \cite[Section~2]{katsurada2026harder}.

Write \(\calO_E=\bbZ[\omega]\),
where \(\omega^2+\omega+1=0\),
and put \(T_1=I_2\) and \(T_2=\diag(1,2)\).
In this specialization,
\begin{equation}\label{eq:example-integral-differential-polynomial}
  \widehat P_{k,j}(T,x)
  =
  \sum_{c=0}^{d}
  (-1)^c
  \binom dc
  \binom{d+j+2}{d-c}
  A_x^{d-c}B_x^{d-c}C_{T,x}^{c}.
\end{equation}
By \cite[Theorem~5.13]{Takeda2025DifferentialOperators},
this differs from the normalized differential polynomial by a nonzero rational scalar whose numerator and denominator involve only primes at most \(2w\).
Hence the two normalizations differ by a \(p\)-adic unit whenever \(p>2w\).

Put \(N=(d+1)^2\).
For \(q\ge1\),
let \(\mathbf t_q(T)\in\bbZ^N\) denote,
in the monomial basis \(u_1^a u_2^{d-a}v_1^b v_2^{d-b}\) with \(0\le a,b\le d\),
the coordinate vector of the following polynomial.
\[
  \sum_{\substack{x\in\calO_E^2\\T[x^*]=q}}
  \widehat P_{k,j}(T,x).
\]

Let \(F_\ell(T,X)\) denote the normalized degree-two local Hermitian Siegel polynomial.
Its functional equation \cite[Corollary~3.2]{Ikeda2008Lifting} gives the following values.
\[
  F_3(T_1,X)=1-3^2X,
  \qquad
  F_2(T_2,X)=1-2^2X,
  \qquad
  F_3(T_2,X)=1+3^2X,
\]
and the remaining local factors are \(1\).
The poset \(\calP_1(T_1)\) consists only of \(\calO_E^2\).
For \(T_2\), there is no enlargement at \(2\), while at \(3\) there are exactly two larger lattices in \(\calP_1(T_2)\),
corresponding to the two isotropic lines of \(\diag(1,-1)\) over \(\bbF_3\).
Thus M\"obius inversion gives the following formulas.
\[
  A_{\Si}^{\prim}(1,T_1)
  \sim_{\frakP}
  1-3^{j+2},
  \qquad
  A_{\Si}^{\prim}(1,T_2)
  \sim_{\frakP}
  (1-2^{j+2})
  \left(1+3^{j+2}-2\cdot3^{j/2+2}\right).
\]
Put \(s_1(j)=1-3^{j+2}\) and \(s_2(j)=(1-2^{j+2})(1+3^{j+2}-2\cdot3^{j/2+2})\).

We put \(B=\lfloor w/6\rfloor\) and \(h=\dim S_{2w}(\SL_2(\bbZ))\).
Define \(\calM_{k,j}\) by the following formula.
\[
  \calM_{k,j}
  =
  \left(
  s_1(j)\mathbf t_q(T_1)
  \ \middle|\
  s_2(j)\mathbf t_q(T_2)
  \right)_{1\le q\le B}
  \in M_{B,2N}(\bbZ).
\]
Define \(\Delta_{k,j}\) by
\[
  \Delta_{k,j}
  =
  \gcd\left\{
  \det M
  \;\middle|\;
  M\text{ is an }h\times h\text{ minor of }\calM_{k,j}
  \right\},
\]
with \(\Delta_{k,j}=0\) if all such minors vanish.

\begin{cor}\label{cor:degree-two-Delta-criterion}
  Assume \(p>2w\),
  \(p\nmid D_E\),
  \(v_p(\zeta(-j-1))=0\),
  and \(\frakd_{\cusp}(1)=\frakA(f^\dagger)=\calO_{L,\frakP}\).
  If \(p\nmid\Delta_{k,j}\),
  then \(\rank_{\bbF_p}(\calM_{k,j}\bmod p)=h\),
  and for some \(i\in\{1,2\}\),
  \[
    v_{\frakP}\!\left(
    A_{\Si}^{\prim}(1,T_i)
    \frac{(\Phi_{1,T_i},f^\dagger)_{G_1}}
    {(f^\dagger,f^\dagger)_{G_1}}
    \right)
    =0.
  \]
\end{cor}

\begin{proof}
  By the formulas above,
  the primitive Siegel coefficients for \(T_1\) and \(T_2\) differ from \(s_1(j)\) and \(s_2(j)\),
  respectively,
  by \(\frakP\)-adic units.
  The scalar relating \(\widehat P_{k,j}\) to the normalized differential polynomial is also a \(\frakP\)-adic unit because \(p>2w\),
  and \(3-2m=-j-1\) identifies the remaining hypothesis with the unit condition for the constant term of the degree-one Eisenstein series.
  Hence \(p\nmid\Delta_{k,j}\) implies that the matrix in Lemma~\ref{lem:degree-two-rank-criterion} has rank \(h\) modulo \(\frakP\),
  and the result follows.
\end{proof}

\begin{ex}\label{ex:7-12-73}
  Take \((k,j,p)=(7,12,73)\).
  Let
  \[
    f_0=q+a_0q^2+\cdots\in S_{24}(\SL_2(\bbZ))
  \]
  be the normalized primitive eigenform for which \(a_0\) is a root of \(X^2-1080X-20468736\).
  Put \(K_0=\bbQ(a_0)\), and \(\frakp=(73,a_0+2)\subset\calO_{K_0}\).
  Choose \(\frakP\) above \(\frakp\) in \(L\),
  and let \(f\) be the corresponding full-level degree-one Hermitian cusp form.
  Then \(m=8\), \(w=12\), \(d=4\), and \(B=h=2\).
  A direct computation gives
  \[
    \Delta_{7,12}=6912=2^8\cdot3^3.
  \]
  We have \(73>2w\), \(73\nmid D_E\), and
  \[
    \zeta(-13)=-\frac1{12}\in\bbZ_{(73)}^\times.
  \]
  The other normalized primitive eigenform in \(S_{24}(\SL_2(\bbZ))\) is the Galois conjugate of \(f_0\),
  whose \(q^2\)-coefficient is \(1080-a_0\).
  Modulo \(\frakp\),
  \[
    a_0\equiv71,
    \qquad
    1080-a_0\equiv60.
  \]
  Hence \(f_0\) is not congruent modulo \(\frakp\) to any other normalized primitive eigenform of weight \(24\).
  Since \(h_E=1\), the preceding remark gives
  \[
    \frakd_{\cusp}(1)
    =
    \frakA(f^\dagger)
    =
    \calO_{L,\frakP}.
  \]
  Therefore Corollary~\ref{cor:degree-two-Delta-criterion} gives an \(i\in\{1,2\}\) such that
  \[
    v_{\frakP}\!\left(
    A_{\Si}^{\prim}(1,T_i)
    \frac{(\Phi_{1,T_i},f^\dagger)_{G_1}}
    {(f^\dagger,f^\dagger)_{G_1}}
    \right)
    =0.
  \]
  The poset \(\calP_1(T_1)\) consists only of \(\calO_E^2\),
  and every \(M\in\calP_1(T_2)\) with \(M\ne\calO_E^2\) differs from \(\calO_E^2\) only at \(3\).
  Hence the local hypothesis at \(73\) in Corollary~\ref{cor:degree-two-primitive-nonvanishing} holds for both \(T_1\) and \(T_2\).
  Applying that corollary with \((\gamma,T)=(1,T_i)\), there exists \(M\in\calP_1(T_i)\).
  Put \(\gamma_0=\gamma[M]\), and \(S_0=T_M\),
  where \(\gamma=1\).
  Then
  \[
    v_{\frakP}\!\left(
    L^{\mathrm{alg}}(15/2,f,\St)
    A_{\Kl}(\gamma_0,S_0)
    \right)
    =0.
  \]
  By \cite[Remark~2.7]{katsurada2026harder},
  \[
    L^{\mathrm{alg}}(15/2,f,\St)
    \sim_{\frakP}
    L^{\mathrm{alg}}(19,f_0)
    L^{\mathrm{alg}}(19,f_0\otimes\chi_{E/\bbQ}).
  \]
  A computation with PARI/GP gives \(v_{\frakP}\!\left(
  L^{\mathrm{alg}}(19,f_0\otimes\chi_{E/\bbQ})
  \right)
  =0\).
  Hence
  \[
    L^{\mathrm{alg}}(15/2,f,\St)
    \sim_{\frakP}
    L^{\mathrm{alg}}(19,f_0),
  \]
  and therefore
  \[
    v_{\frakP}\!\left(
    L^{\mathrm{alg}}(19,f_0)
    A_{[f_0]^2_{16}}(\gamma_0,S_0)
    \right)
    =0.
  \]
\end{ex}

The condition \(p\nmid\Delta_{k,j}\) is sufficient,
but not necessary.

\begin{ex}\label{ex:4-24-97}
  Take \((k,j,p)=(4,24,97)\).
  Then \(m=14\),
  \(w=15\),
  \(d=1\),
  and \(B=h=2\).
  A direct computation gives
  \[
    \Delta_{4,24}
    =
    2^{10}\cdot3^3\cdot47\cdot97\cdot383\cdot2731\cdot8191\cdot181967.
  \]
  Thus \(97\mid\Delta_{4,24}\),
  and Corollary~\ref{cor:degree-two-Delta-criterion} does not apply.
  Let
  \[
    f_0=q+a_0q^2+\cdots\in S_{30}(\SL_2(\bbZ))
  \]
  be the normalized primitive eigenform with \(a_0=4320+96\sqrt{51349}\).
  Put \(K_0=\bbQ(\sqrt{51349})\),
  and \(\frakp=(97,\sqrt{51349}-6)\subset\calO_{K_0}\).
  Choose \(\frakP\) above \(\frakp\) in \(L\),
  and let \(f\) be the full-level Hermitian cusp form on \(G_1\) corresponding to \(f_0\).
  Modulo \(\frakp\),
  \[
    a_0\equiv46.
  \]
  Since \(\dim S_{30}(\SL_2(\bbZ))=2\),
  the other normalized primitive eigenform is the Galois conjugate of \(f_0\).
  Its \(q^2\)-coefficient satisfies
  \[
    4320-96\sqrt{51349}
    \equiv58\pmod{\frakp},
  \]
  so \(f_0\) is not congruent modulo \(\frakp\) to any other normalized primitive eigenform of weight \(30\).
  Moreover,
  \[
    \zeta(-25)
    =
    -\frac{8553103}{156}
    \in\bbZ_{(97)}^\times.
  \]
  Since \(h_E=1\),
  the preceding remark gives \(\frakd_{\cusp}(1)=\frakA(f^\dagger)=\calO_{L,\frakP}\).

  Let \(H\in S_{30}(\SL_2(\bbZ))\) be the \(u_1v_1\)-coordinate of \(\Phi_{1,I_2}\).
  Direct calculation gives
  \[
    H\equiv59q+95q^2\pmod{\frakp}.
  \]
  Since \(95\cdot59^{-1}\equiv46\pmod{97}\),
  the first two Fourier coefficients of \(H-59f_0\) vanish modulo \(\frakp\).
  The Sturm bound in weight \(30\) is \(2\),
  and hence
  \[
    H\equiv59f_0\pmod{\frakp}.
  \]
  Since \(\frakA(f^\dagger)=\calO_{L,\frakP}\), the map
  \[
    h\longmapsto
    \frac{(h,f^\dagger)_{G_1}}
    {(f^\dagger,f^\dagger)_{G_1}}
  \]
  is integral on the cuspidal lattice.
  Under this identification, the congruence above therefore gives
  \[
    \frac{(H,f^\dagger)_{G_1}}
    {(f^\dagger,f^\dagger)_{G_1}}
    \equiv59\pmod{\frakP}.
  \]
  Thus the \(u_1v_1\)-coordinate of \((\Phi_{1,I_2},f^\dagger)_{G_1}/(f^\dagger,f^\dagger)_{G_1}\) is a \(\frakP\)-adic unit.
  Moreover,
  \[
    A_{\Si}^{\prim}(1,I_2)
    \sim_{\frakP}
    1-3^{26}
    \equiv10\not\equiv0\pmod{\frakP}.
  \]
  Therefore
  \[
    v_{\frakP}\!\left(
    A_{\Si}^{\prim}(1,I_2)
    \frac{(\Phi_{1,I_2},f^\dagger)_{G_1}}
    {(f^\dagger,f^\dagger)_{G_1}}
    \right)
    =0.
  \]
  Since \(\calP_1(I_2)\) consists only of \(\calO_E^2\),
  Corollary~\ref{cor:degree-two-primitive-nonvanishing} therefore gives
  \[
    v_{\frakP}\!\left(
    L^{\mathrm{alg}}(27/2,f,\St)
    A_{[f_0]^2_{28}}(1,I_2)
    \right)
    =0.
  \]
  As in the preceding example, we have
  \[
    L^{\mathrm{alg}}(27/2,f,\St)
    \sim_{\frakP}
    L^{\mathrm{alg}}(28,f_0).
  \]
  Thus
  \[
    v_{\frakP}\!\left(
    L^{\mathrm{alg}}(28,f_0)
    A_{[f_0]^2_{28}}(1,I_2)
    \right)
    =0.
  \]
\end{ex}

\bibliographystyle{plain}
\bibliography{primitive_hermitian_klingen}

\end{document}